\documentclass[bj]{imsart}

\RequirePackage{amsthm,amsmath,amsfonts,amssymb}
\RequirePackage[numbers]{natbib}

\RequirePackage[colorlinks,citecolor=blue,urlcolor=blue]{hyperref}

\startlocaldefs
\numberwithin{equation}{section}
\theoremstyle{plain}

\newtheorem{theorem}{Theorem}[section]
\newtheorem{lemma}[theorem]{Lemma}
\newtheorem{proposition}[theorem]{Proposition}
\newtheorem{corollary}[theorem]{Corollary}
\theoremstyle{remark}
\newtheorem{remark}[theorem]{Remark}

\newtheorem{fact}[theorem]{Fact}

\newcommand{\N}{\mathbb{N}}
\newcommand{\Z}{\mathbb{Z}}
\newcommand{\R}{\mathbb{R}}
\newcommand{\C}{\mathbb{C}}

\newcommand{\ind}[1]{\mathbf{1}_{\left\{#1\right\}}}

\renewcommand{\bar}[1]{\overline{#1}}
\renewcommand{\tilde}[1]{\widetilde{#1}}
\renewcommand{\hat}[1]{\widehat{#1}}

\newcommand{\dd}{\mathrm{d}}

\DeclareMathOperator{\E}{\mathbb{E}}
\newcommand{\Q}{\mathbb{Q}}
\renewcommand{\P}{\mathbb{P}}

\newcommand{\calP}{\mathcal{P}}

\newcommand{\x}{\mathbf{x}}

\newcommand{\X}{\mathbf{X}}
\newcommand{\bfZ}{\mathbf{Z}}

\renewcommand{\rho}{\varrho}
\renewcommand{\epsilon}{\varepsilon}

\endlocaldefs

\begin{document}

\begin{frontmatter}
\title{A necessary and sufficient condition for the convergence of the derivative martingale in a branching Lévy process}
\runtitle{The derivative martingale in a branching Lévy process}

\begin{aug}
\author[A]{\fnms{Bastien} \snm{Mallein}\ead[label=e1]{mallein@math.univ-paris13.fr}}
\and
\author[B]{\fnms{Quan} \snm{Shi}\ead[label=e2,mark]{quan.shi@amss.ac.cn}}
\address[A]{Université Sorbonne Paris Nord, LAGA, UMR 7539, F-93430, Villetaneuse, France, \printead{e1}}

\address[B]{Academy of Mathematics and Systems Science, Chinese Academy of Sciences, 100190,
Beijing, China, \printead{e2}}
\end{aug}

\begin{abstract}
A continuous-time particle system on the real line satisfying the branching property and an exponential integrability condition is called a branching Lévy process, and its law is characterized by a triplet $(\sigma^2,a,\Lambda)$. 
We obtain a necessary and sufficient condition for the convergence of the derivative martingale of such a process to a non-trivial limit in terms of $(\sigma^2,a,\Lambda)$. This extends previously known results on branching Brownian motions and branching random walks.
To obtain this result, we rely on the spinal decomposition and establish a novel zero-one law on the perpetual integrals of centred Lévy processes conditioned to stay positive.
\end{abstract}



\end{frontmatter}


\section{Introduction}

A branching random walk is a discrete-time particle system on the real line, which can be constructed as follows. It starts with an initial individual at position $0$. This individual gives birth at time $1$ to children, that are positioned according to a certain point process. Then each child gives birth at time $2$ to offspring positioned around their parent according to an i.i.d.\@ copy of that point process. The process continues, each generation of individuals giving birth independently to children positioned around their parent, according to a shifted copy of the same point process.

To describe this process, we introduce some notation. We denote by
\[
  \mathcal{P} = \left\{ \x=(x_1, x_2,\ldots) \in (-\infty,\infty]^\N : x_1\le x_2\le \cdots \text{ and } \lim_{n \to \infty} x_n = \infty \right\}
\]
the space of non-decreasing sequences $\x$ on $(-\infty,\infty]$ that converge to $\infty$. Equivalently, these sequences can be identified with Radon point measures $\mu$ on $\R$ with finite mass on $(-\infty,0]$ through the identification
\[
  \mu = \sum_{n \geq 1} \ind{x_n < \infty}\delta_{x_n} \quad \iff \quad \x = \big(\inf\left\{  y \in \R : \mu((-\infty,y]) \geq n \right\}, n \geq 1\big).
\]
In particular, if $\mu = \sum_{j=1}^n x_j$, it is identified with the element $\x = (x_1,\ldots x_n,\infty,\infty,\cdots) \in \mathcal{P}$. In other words, with this notation the point $\infty$ is taken as a cemetery state, and the Dirac mass $\delta_\infty$ is identified to $0$. For $y \in \R$, we denote by $\tau_y \x = (x_n + y, n \geq 1)$ the translation operator on $\mathcal{P}$.

Then, we can describe the branching random walk as a $\mathcal{P}$-valued Markov process $(\X_n, n \geq 0)$ satisfying the branching property: for all $0 \leq k \leq n$, setting $\x = \X_k$, we have
\begin{equation}
\label{eqn:discreteBranching}
\X_{n} = \sum_{j=1}^\infty \tau_{x_j} \X^{(j)}_{n-k} \quad \text{ in law, where } \X^{(j)}_{n-k} \text{ are i.i.d.\ copies of } \X_{n-k}.
\end{equation}
Observe that if there exists $\theta \geq 0$ such that $c = \E\left(\int e^{-\theta x}\X_1(\dd x)\right)< \infty$, then we have $\X_k \in \mathcal{P}$ for all $k \in \N$ a.s.; indeed, by the branching property we then have $\E\left(\int e^{-\theta x}\X_k(\dd x)\right) = c^k$.

Branching Lévy processes are the continuous-time counterparts of branching random walks. A branching Lévy process (with possibly infinite birth intensity) is a continuous-time particle system on the real line, in which particles move according to i.i.d.\@ Lévy processes, and reproduce in a Poisson fashion, possibly on an everywhere dense countable set of times. Branching Lévy processes were introduced in \cite{Ber16}, as an intermediate tool to study homogeneous growth-fragmentation processes. They can be constructed as increasing limits of Uchiyama-type continuous-time branching random walks, introduced in \cite{Uch82}.

In \cite{BeM18}, it is proved that branching Lévy processes are the unique càdlàg (for the topology of vague convergence) $\mathcal{P}$-valued processes $\bfZ$ that satisfy the following two properties. 
\begin{description}
	\item[Branching property:] for all $0 \leq s \leq t$, setting $\mathbf{z}=\bfZ_s$,
	\begin{equation}
	\label{eqn:branching}
	\bfZ_t = \sum_{j=1}^\infty \tau_{z_j} \bfZ^{(j)}_{t-s} \text{ in law, where }\bfZ^{(j)}_{t-s} \text{ are i.i.d.\ copies of } \bfZ_{t-s}. 
	\end{equation}
	\item[Exponential integrability:] there exists $\theta > 0$ such that for all $t \geq 0$,
	\begin{equation}
	\label{eqn:expIntegrabilityBase}
	\E\left( \int e^{-\theta x} \bfZ_t(\dd x) \right) < \infty.
	\end{equation}
\end{description}
The class of branching Lévy processes generalizes several classical models, including the branching Brownian motion, continuous-type branching random walks \cite{Uch82}, branching L\'evy processes with finite birth intensity \cite{Kyp99}, and, up to an exponential transform, homogeneous fragmentations \cite{Ber01} and compensated growth-fragmentations \cite{Ber16}.
Observe that the branching property implies that for all $t>0$, the process $(\bfZ_{nt}, n \geq 0)$ is a branching random walk.

The law of a branching Lévy process is characterized by a triplet $(\sigma^2, a, \Lambda)$, with $\sigma^2 \geq 0$, $a \in \R$ and $\Lambda$ a sigma-finite measure on $\mathcal{P}$ with the following integrability conditions. We assume that
\begin{equation}
\label{eqn:levyEve}
\int_{\calP} 1 \wedge x_1^2 \Lambda(\dd \x) < \infty,
\end{equation}
and that there exists $\theta \geq 0$ satisfying
\begin{equation}
\label{eqn:expIntegrability}
\int_{\calP} \Big(e^{-\theta x_1} \ind{x_1 < - 1} + \sum_{j \geq 2} e^{-\theta x_j}\ind{x_j > -\infty}\Big) \Lambda(\dd \x) < \infty.
\end{equation}
Informally, the branching Lévy process can be described as follows: each particle moves independently according to a Lévy process with diffusion coefficient~$\sigma$, drift $a$, and Lévy jump measure the image of $\Lambda$ by the application $\x \mapsto x_1$. Simultaneously, the particle gives birth to children in a Poissonian fashion, in such a way that at rate $\Lambda(\dd \x)$, a particle makes a jump of $x_1$ while simultaneously giving birth to offspring positioned at distances $x_2,x_3,\ldots$ from the pre-jump position of the parent.

The set of individuals in a branching Lévy process can be indexed by using a generalization of the Ulam--Harris notation, introduced by Shi and Watson in \cite[Section 2]{ShW19} for the study of growth-fragmentation processes. This enumeration allows the description of the genealogy of particles in a branching Lévy process. For each $t \geq 0$, we denote by $\mathcal{N}_t$ the set of individuals alive at time $t$, and for all $s \leq t$, by $X_s(u)$ the position at time $s$ of a particle $u\in \mathcal{N}_t$ or its ancestor, if the particle is not born yet at time $s$.

Condition \eqref{eqn:levyEve} ensures that the trajectories of particles are well-defined Lévy processes. Condition \eqref{eqn:expIntegrability} implies \eqref{eqn:expIntegrabilityBase} (see \cite[Theorem 1.2(ii)]{BeM18}), which ensures that for all $t \!\geq\! 0$ the random measure $\sum_{u \in \mathcal{N}_t} \delta_{X_t(u)}$ almost surely belongs to $\mathcal{P}$.

\begin{remark}
	The definition we give here of a branching Lévy process differs from the ones in \cite{BeM18,BeM18b}. In those articles, the branching Lévy processes were constructed in such a way that they have finite mass on $[0,\infty)$ a.s., instead of $(-\infty,0]$. We choose to change this convention in order to be consistent with the corresponding results for branching random walks in the boundary case, described in e.g.\ \cite{BiK04}.
\end{remark}

Under assumption \eqref{eqn:expIntegrability},  for all $z \in \C$ with $\mathfrak{R}(z) = \theta$ and $t \geq 0$, we have by \cite[Theorem 1.2(ii)]{BeM18}
\begin{equation}
\label{eqn:cumulantProp}
\E\bigg( \sum_{u \in \mathcal{N}_t} e^{-z X_t(u)} \bigg) = \exp(t \kappa(z)),
\end{equation}
where
\begin{equation}
\label{eqn:cumulantDef}
\kappa(z) := \frac{\sigma^2 z^2}{2} - a z + \int_{\mathcal{P}} \sum_{j \geq 1} e^{-z x_j} - 1 + z x_1 \ind{|x_1| < 1} \Lambda(\dd \x)
\end{equation}
is called the cumulant generating function of the branching Lévy process. Equation \eqref{eqn:cumulantProp} and the branching property imply that 
\[
W_t(\theta) := \sum_{u \in \mathcal{N}_t} e^{-\theta X_t - t \kappa(\theta)}, \qquad t\ge 0
\]
is a non-negative martingale. As such, its limit $W_\infty := \lim_{t \to \infty} W_t$ exists a.s.\ and is finite. Whether the terminal value is degenerate (i.e.\ $W_\infty(\theta) = 0$ a.s.\@) or not has fundamental importance in the study of branching processes, and has been investigated by a considerable amount of literature.

Under the condition that $\kappa'(\theta) := \frac{1}{i}\frac{\dd }{\dd \xi} \kappa(\theta + i \xi)|_{\xi =0}$ exists and is finite, it is now well-known that
\begin{equation}
\label{eqn:supercriticalMartingale}
\theta \kappa'(\theta) < \kappa(\theta)
\end{equation}
is a necessary condition for the non-degeneracy of $W_\infty(\theta)$, and that, up to an additional integrability assumption on $\Lambda$, this condition is also sufficient. This result was first proved by Bigggins \cite{Big77} in the context of branching random walks, by McKean \cite{McK75} (see also Neveu \cite{Neu87}) for the branching Brownian motion (which is the only branching Lévy process with continuous trajectories).

Lyons \cite{Lyo97} then gave an elementary proof of this statement, by introducing the spinal decomposition, a critical tool in the study of branching processes. Bertoin and Mallein \cite{BeM18b} then adapted the proof of Lyons to the settings of branching Lévy processes. A necessary and sufficient condition for the non-degeneracy of $W_\infty$ regardless of the existence of $\kappa'(\theta)$ was obtained by Alsmeyer and Iksanov \cite{AlI09} for branching random walks, and by Iksanov and Mallein \cite{IkM18} for branching Lévy processes, by using result on the finiteness of perpetual integrals. 

In the boundary case, when \eqref{eqn:supercriticalMartingale} fails to hold, the role of the additive martingale is replaced by a different martingale, which is not uniformly integrable. More precisely, if there exists $\theta^* > 0$ such that
\begin{equation}
\label{eqn:criticalMartingale}
\theta^* \kappa'(\theta^*) = \kappa(\theta^*),
\end{equation}
then the martingale $W_t(\theta^*)$ converges to $0$ almost surely, as stated above. However, the process
\begin{equation}
\label{eqn:derivativeMartingale}
Z_t := \sum_{u \in \mathcal{N}_t} (\theta^* X_t(u) + t \kappa(\theta^*)) e^{-\theta^* X_t(u) - t \kappa(\theta^*)},\qquad t\ge 0
\end{equation}
is a signed martingale, often called the derivative martingale. The almost sure limit $Z_\infty = \lim_{t \to \infty} Z_t$, if existing, is non-negative.

Assuming in addition that 
\begin{equation}
\label{eqn:finiteVariance}
\kappa''(\theta^*) := \E\bigg( \sum_{u \in \mathcal{N}_1} \left(X_1(u) + \kappa'(\theta^*)\right)^2 e^{-\theta^* X_1(u) - \kappa(\theta^*)} \bigg) \in (0,\infty),
\end{equation}
Aïdékon \cite{Aid13} obtained sufficient integrability conditions for the non-de\-ge\-ne\-ra\-cy of $Z_\infty$ for branching random walks. These conditions were shown to be necessary by Chen \cite{Che15}, who proved that if they do not hold, the derivative martingale converges almost surely to $0$. For the branching Brownian motion, the optimal condition for the non-degeneracy of the limit of the derivative martingale was previously obtained by Yang and Ren \cite{YaR11}. It is worth noting that with a proper rescaling, the critical martingale $W_t$ will converge to $Z_\infty$ (see e.g.\ \cite{AiS14} for a proof in the context of branching random walks and \cite{BBM18} for a continuous-time extension).

When $Z_t$ converges to a non-degenerate limit $Z_\infty$, the random variable $Z_\infty$ is related to the maximal displacement of the branching process. Lalley and Sellke \cite{LaS87} showed a deep connection between the derivative martingale and the asymptotic behaviour of the maximal displacement of particles for the branching Brownian motion. More precisely, they showed that for all $y \in \R$, 
\[
\lim_{t \to \infty} \P\left(\max_{u \in \mathcal{N}_t} X_t(u) \leq \sqrt{2} t - \frac{3}{2\sqrt{2}} \log t + y \right) = \E\left( \exp\left( - C Z_\infty e^{-\sqrt{2}y}\right) \right),
\]
famously correcting an error in \cite{McK75}, who used the critical additive instead of the derivative martingale (see the survey of Biggins and Kyprianou \cite{BiK05} on that subject). The convergence in law of the maximum of a branching random walk was proved by Aïdékon \cite{Aid13}, which is again related to the derivative martingale. This result was then extrapolated to branching Lévy processes by Dadoun \cite{Dad17}. To sum up, obtaining necessary and sufficient conditions for the derivative martingale limit to be non-trivial is relevant to understand the asymptotic properties of extremal particles in branching processes.

For branching Lévy processes with one-sided jumps, Shi and Watson \cite{ShW19} obtained sufficient conditions for the convergence of the derivative martingale, by adapting spinal decomposition arguments dating back to \cite{LPP95,BiK04}. 
The aim of this work is to obtain optimal integrability conditions for general branching Lévy processes, which is, however, a much harder question.  Note that since $(\bfZ_n, n \geq 0)$ is a branching random walk, the optimal integrability conditions of Aïdékon and Chen allow us to immediately obtain a necessary and sufficient condition for the non-degeneracy of $Z_\infty$ in terms of the law of $\bfZ_1$. However, this condition does not translate easily in terms of conditions on $(\sigma^2,a,\Lambda)$, as the connection between the two quantities is intricate.

As a result, we instead refine the spinal decomposition method in \cite{ShW19} to prove the non-degeneracy of $Z_\infty$, for which we still have to overcome some substantial challenges. 
The problem relies on the analysis of a Lévy process $\xi$ conditioned to stay positive, in particular, a necessary and sufficient condition for the finiteness of $\int_0^{\infty} f(\xi_t) \dd t$ (a so-called \emph{perpetual integral}), which is, to our knowledge, not available in the literature.  
We build in the forthcoming Proposition~\ref{prop:finiteness} a novel zero-one law for the finiteness of such perpetuities of Lévy processes conditioned to be positive, by using new techniques from \cite{KoS19,BDK20}. This is another main contribution of this work and might be of independent interest.

\section{Main results}

The aim of this article is to obtain necessary and sufficient conditions for the non-degeneracy of $Z_\infty$, the limit of the derivative martingale in a branching Lévy process. We first observe that up to a space-time linear transform, the branching Lévy process can be assumed to be in the so-called \emph{boundary case}, which will simplify computations and notation later on.

More precisely, assuming \eqref{eqn:criticalMartingale} and \eqref{eqn:finiteVariance}, set for all $t \geq 0$, $u \in \mathcal{N}_t$ and $s \leq t$
\[
Y_s(u) = \theta^* X_s(u) + s \kappa(\theta^*),
\]
with $\theta^*$ the unique positive solution of the equation $\theta \kappa'(\theta) - \kappa(\theta) = 0$. One immediately obtains that $Y$ is a branching Lévy process with characteristic triplet $(\sigma^2_{\theta^*}, a_{\theta^*}, \Lambda_{\theta^*})$ which can be expressed explicitly in terms of $(\sigma^2, a, \Lambda)$ and $\theta^*$. Moreover, $\Lambda_{\theta^*}$ satisfies \eqref{eqn:levyEve} and \eqref{eqn:expIntegrability} with $\theta = 1$, and we have 
\begin{align*}
\log \E\bigg( \sum_{u \in \mathcal{N}_1} e^{-Y_1(u)} \bigg) &= \log \E\bigg( \sum_{u \in \mathcal{N}_1} e^{-\theta^* X_1(u)} \bigg) - \kappa(\theta^*) = 0,\\
\E\bigg( \sum_{u \in \mathcal{N}_1} Y_1(u)e^{-Y_1(u)} \bigg) &= \E\bigg( \sum_{u \in \mathcal{N}_1}(\theta^* X_1(u) + \kappa(\theta^*)) e^{-\theta^* X_1(u)- \kappa(\theta^*)} \bigg)\\
&= \theta^* \kappa'(\theta^*) - \kappa(\theta^*) = 0,\\
\text{and }\ \E\bigg( \sum_{u \in \mathcal{N}_1} Y_1(u)^2 e^{-Y_1(u)} \bigg) &= \E\bigg( \sum_{u \in \mathcal{N}_1} (\theta^* X_1(u) + \kappa(\theta^*))^2 e^{-\theta^* X_1(u) - \kappa(\theta^*)} \bigg) \\ 
& =  (\theta^*)^2\kappa''(\theta^*) \in (0, \infty).
\end{align*}

Therefore, up to a linear transformation, we can assume without loss of generality that the branching Lévy process $X$ satisfies
\begin{equation}
\label{eqn:H0}
\E\bigg( \sum_{u \in \mathcal{N}_1} e^{-X_1(u)} \bigg) = 1, \quad \E\bigg( \sum_{u \in \mathcal{N}_1} X_1(u) e^{-X_1(u)} \bigg) = 0.
\end{equation}
Then we say that $X$ is in \emph{the boundary case}, in the language of \cite{BiK05}.
The assumption \eqref{eqn:H0} is classical when studying spatial branching processes around their critical parameter. 
A classification of branching random walks that can be assumed to be in the boundary case is given in the appendix of the arXiv version of \cite{Jaf}, or in \cite{BeG11}. Moreover, we also assume that 
\begin{equation}
\label{eqn:var}
\E\bigg( \sum_{u \in \mathcal{N}_1} X_1(u)^2 e^{-X_1(u)} \bigg) \in (0,\infty),
\end{equation}
which is a second moment integrability condition, equivalent to \eqref{eqn:finiteVariance}.

Note that a branching Lévy process $X$ with characteristic triplet $(\sigma^2,a,\Lambda)$ satisfies \eqref{eqn:H0}, i.e. is in the boundary case, if and only if $(\sigma^2,a, \Lambda)$ satisfies
%
\begin{align}
\begin{cases} 
&\displaystyle\frac{\sigma^2}{2} = \displaystyle \int_{\mathcal{P}} \bigg( \Big(\sum_{j \geq 1}  (1 + x_j) e^{-x_j}\Big) - 1\bigg) \Lambda(\dd \x)\\
\label{eqn:H0bis}
&\displaystyle  a = \displaystyle \frac{\sigma^2}{2} + \int_{\mathcal{P}} \bigg( \Big( \sum_{j \geq 1} e^{-x_j} \Big) - 1 + x_1 \ind{|x_1| < 1} \bigg) \Lambda(\dd \x),
\end{cases}
\end{align}
or equivalently $\kappa(1) = \kappa'(1)=0$, in terms of the cumulant generating function $\kappa$ defined in \eqref{eqn:cumulantDef}. Similarly, \eqref{eqn:var} can be rewritten as 
\begin{equation}
\label{eqn:varbis}
\int_{\calP} \sum_{i \ge 1} x_i^2 e^{-x_i}\Lambda(d \x) < \infty \iff \kappa''(1) \in (0, \infty).
\end{equation}
We now state our main result.

\begin{theorem}
	\label{thm:main}
	Let $(X_t(u), u \in \mathcal{N}_t)_{t \geq 0}$ be a branching Lévy process satisfying \eqref{eqn:H0} and \eqref{eqn:var}, and $Z_t= \sum_{u \in \mathcal{N}_t} X_t(u) e^{-X_t(u)}$ for $t\ge 0$. The derivative martingale $(Z_t)_{t\ge 0}$ converges a.s.\@ to a non-negative non-trivial limit $Z_{\infty}$ if and only if
	\begin{equation}
	\label{H1}
	\int_{\calP} \left( Y(\x) \log_+(Y(\x)-1)^2 + \tilde{Y}(\x) \log_+(\tilde{Y}(\x) - 1) \right) \Lambda(d \x) < \infty, \tag{H}
	\end{equation}
	where $\log_+ : x \in [0,\infty) \mapsto \max(0,\log(x))$ and for any $\x \in \calP$, we have set
	\[
	Y(\x):= \sum_{i=1}^\infty e^{-x_i}
	\quad \text{and} \quad 
	\tilde{Y}(\x):= \sum_{i = 1}^{\infty} \ind{x_i\ge 0} x_i e^{-x_i}.
	\]
\end{theorem}

\begin{remark}
	\label{rem:nonTrivial}
	By the branching property of a branching Lévy process, if $Z_\infty$ exists, then the event $\{Z_\infty \!=\! 0\}$ is an inherited property of the underlying Galton-Watson tree (see \cite[Discussion 5.4]{ShiSF15} for a related argument for the discrete-time branching random walk). As a result, either $Z_\infty = 0$ a.s.\@ or $Z_\infty > 0$ a.s.\@ on the survival set $\{\mathcal{N}_t\ne \emptyset, \forall t\ge 0\}$ of the branching Lévy process.
\end{remark}

As observed above, the convergence of the derivative martingale in a generic branching Lévy process can be obtained from Theorem~\ref{thm:main}. We state here the result for a general branching L\'evy process, without assuming the boundary condition.
\begin{theorem}
	\label{thm:main-general}
	Let $(X_t(u), u \in \mathcal{N}_t)_{t \geq 0}$ be a branching Lévy process satisfying \eqref{eqn:criticalMartingale} and \eqref{eqn:finiteVariance}, and $Z_t=  \sum_{u \in \mathcal{N}_t} (\theta^* X_t(u) + t \kappa(\theta^*)) e^{-\theta^* X_t(u) - t \kappa(\theta^*)}$ for $t\ge 0$. The derivative martingale $(Z_t)_{t\ge 0}$ converges a.s.\@ to a non-negative non-trivial limit $Z_{\infty}$ if and only if
	\begin{equation}
		\label{H1bis}
		\int_{\calP} \left( Y(\x) \log_+(Y(\x)-1)^2 + \tilde{Y}(\x) \log_+(\tilde{Y}(\x) - 1) \right) \Lambda(d \x) < \infty, \tag{H*}
	\end{equation}
	where $\log_+ : x \in [0,\infty) \mapsto \max(0,\log(x))$ and for any $\x \in \calP$, we have set
	\[
	Y(\x):= \sum_{i=1}^\infty e^{-\theta^* x_i}
	\quad \text{and} \quad 
	\tilde{Y}(\x):= \sum_{i = 1}^{\infty} \ind{x_i\ge 0} \theta^* x_i e^{- \theta^* x_i}.
	\]
\end{theorem}
\begin{proof}
	Consider the branching Lévy process 
$
	(X'_s(u) := \theta^* X_s(u) + s \kappa(\theta^*), t\ge 0, u \in \mathcal{N}_t, s\le t). 
$
In particular, $X'$ satisfies the boundary condition and its L\'evy measure $\Lambda'$ is the pushforward of $\Lambda$ via the function $\x\mapsto \theta^* \x$. 
By Theorem~\ref{thm:main}, the derivative martingale of $X'$, which is identical to that of $X$,  converges a.s.\@ to a non-negative non-trivial limit if and only if \eqref{H1} holds for $\Lambda'$. Rewriting this condition in terms of $\Lambda$ leads to the result.  
\end{proof}

As we consider a general class of branching L\'evy processes, Theorem~\ref{thm:main} generalizes a few sufficient conditions previously obtained in the literature: \cite{BeR05} for fragmentations, \cite{ShW19} for growth-fragmentations, and \cite{CheShe} for branching L\'evy processes with finite birth rate. 
It also generalizes the necessary and sufficient condition obtained by \cite{YaR11} for branching Brownian motion. 

The rest of the article is organised as follows. We introduce some well-known fact on Lévy processes conditioned to stay positive in the next section, and establish a necessary and sufficient condition for finiteness of associated perpetual integrals. In Section~\ref{sec:truncatedMartingale}, we introduce the spinal decomposition of the branching Lévy process associated to the additive and the derivative martingales. In Section~\ref{sec:sufficient}, we prove that under assumptions \eqref{eqn:H0} and \eqref{eqn:var}, \eqref{H1} implies that the derivative martingale converges to a non-degenerate limit using the same classical argument as in \cite{Lyo97}. Finally, we prove the ``necessary part'' of Theorem~\ref{thm:main} in Section~\ref{sec:necessary}.

\section{The many-to-one lemma and Lévy processes conditioned to stay positive}
\label{sec:manyToOne}

We present in this section the many-to-one lemma, that links first moments of additive functionals of a branching Lévy process with the law of an associated Lévy process. Then, we introduce some estimates on Lévy processes, in particular defining the law of the Lévy process conditioned to stay positive. We end this section by establishing a novel necessary and sufficient condition for the finiteness of a perpetual integral of the Lévy process conditioned to stay positive. 

\subsection{The many-to-one lemma}
\label{subsec:manyToOne}

One can observe that for all $r \in \R$, the function
$
\Psi : r \in \R \mapsto \kappa(1 + ir)
$
can be rewritten as
\begin{equation}
\label{eqn:defPsi}
\Psi(r) = -\frac{\sigma^2}{2} r^2 + i \hat{a} r + \int_{\R} e^{i r x} - 1 + r x \ind{|x|<1} \pi(\dd x),
\end{equation}
where
\begin{equation}
\label{eqn:defhata}
\hat{a} = a - \sigma^2 + \int_{\mathcal{P}} \sum_{j \geq 1} x_j e^{-x_j} \ind{|x_j|<1} - x_1 \ind{|x_1|<1} \Lambda(\dd \x), 
\end{equation}
and $\pi$ is the sigma-finite measure on $\R$ satisfying for all measurable non-negative functions $f$:
\[
\int_{\R} f(x) \pi (\dd x) = \int_{\mathcal{P}} \sum_{j \geq 1} f(x_j) e^{- x_j} \Lambda(\dd \x).
\]
In other words, the function $\Psi$ can be seen as the Lévy-Khinchine exponent of a Lévy process $\xi$ with diffusion term $\sigma^2$, drift $\hat{a}$ and jump measure $\pi$. 
This fact can be related to the celebrated many-to-one lemma in the context of branching random walks, which can be tracked back at least to the work of Kahane and Peyrière \cite{Pey74,KaP76}: 
roughly speaking, the many-to-one lemma links additive moments of a branching random walk with random walk estimates.
\begin{lemma}[Many-to-one lemma]
	\label{lem:manytoone}
	Let $(X_t(u), u \in \mathcal{N}_t)_{t \geq 0}$ be a branching Lévy process satisfying \eqref{eqn:H0} and $\xi$ a Lévy process with Lévy-Khin\-chi\-ne exponent $\Psi$. For any non-negative measurable function $f$ and $t \geq 0$, we have
	\[
	\E\bigg( \sum_{u \in \mathcal{N}_t} f(X_s(u), s \leq t) \bigg) = \E\left( e^{\xi_t} f(\xi_s, s \leq t) \right).
	\]
\end{lemma}

We refer to \cite[Lemma 2.2]{BeM18} for a proof in branching Lévy processes settings, and to the proof of \cite[Lemma 2.2]{BeM18b} for the computation of $\Psi$. The many-to-one lemma can be thought of as a preliminary version of the spinal decomposition, that we describe in Section~\ref{sec:spinalDecompositionProper}.

\subsection{Lévy processes conditioned to stay positive}\label{sec:levy-cond}

In this section, we denote by $\xi$ a Lévy process with Lévy-Khinchine exponent $\Psi$ given by \eqref{eqn:defPsi}. By \eqref{eqn:H0}, \eqref{eqn:var} and the many-to-one lemma, this Lévy process is centred with finite variance. In particular, it is oscillating, i.e.
\[
\limsup_{t \to \infty} \xi_t = \limsup_{t \to \infty} -\xi_t = \infty \qquad \text{a.s..}
\]
We recall in this section the definition of the law of $\xi$ conditioned to stay above a given level. We refer to \cite[Section 2]{ChD05} for a self-contained construction of this law. 

For all $x \in \R$, we denote by $\tau_x := \inf\{ t \geq 0 : \xi_t < x\}$ the first passage time below level $x$ of $\xi$, and $\P_x$ the law of $(x+\xi_t, t \geq 0)$, a Lévy process with Lévy-Khinchine exponent $\Psi$ starting from $x$. To simplify notation, we also write $\tau := \tau_0$ and $\P := \P_0$. We introduce the renewal function associated to this process, defined for all $x \geq 0$ by
\begin{equation}
\label{eqn:renewalFunction}
R(x) = \E\left(\int_0^\infty \ind{\tau_{-x} > t} \dd L_t \right),
\end{equation}
where $L$ is the local time at zero of the reflected process $(\xi_t- \inf_{s \leq t} \xi_s, t \geq 0)$. We extend this definition by setting $R(x) = 0$ for all $x < 0$. Since $\xi$ does not drift to $-\infty$ (i.e.\ it is not the case that $\lim_{t \to \infty} \xi_t =-\infty$ a.s.), we know from \cite[Lemma~1]{ChD05} that the function $R$ satisfies for all $x \in \R$ and $t \geq 0$
\begin{equation}
\label{eqn:formulaRenewal}
R(x) = \E\left( R(\xi_t + x) \ind{\tau_{-x} > t}  \right) = \E_x(R(\xi_t) \ind{\tau>t}).
\end{equation}

Let us recall some additional properties of $R$. If $0$ is a regular point for the Lévy process $\xi$ (i.e.\ $\inf\{t>0\colon \xi_t = 0\} = 0$ a.s.), then $R(0)=0$, otherwise we normalize $L$ such that $R(0)=1$. Moreover, the function $R$ is finite, continuous, increasing, and $x \mapsto R(x)-R(0)$ is sub-additive on $[0,\infty)$. As $\xi$ has zero mean and finite variance, there exist $0 < c_1 < c_\star < c_2< \infty$ such that
\begin{equation}
\label{eqn:estimatesRenewal}
c_1 x \leq R(x) \leq c_2 (x+1) \quad \text{and} \lim_{x \to \infty} \frac{R(x)}{x} = c_\star.
\end{equation}
See \cite[Theorem~I.21]{Ber96} and \cite[Theorem~7]{DoM02}.

By \eqref{eqn:formulaRenewal}, the process $\big(\!\frac{R(\xi_t+x)}{R(x)} \ind{\tau_{-x} > t}, t \geq 0 \big)$ is a non-negative $\P$-martingale. 
Let $(\mathcal{F}_t)$ be the filtration associated to $\xi$. For all $x > 0$, we denote by $\P^\uparrow_x$ the law defined for all $t \geq 0$ by
\begin{equation}
\label{eqn:lawConditioned}
\left. \P^\uparrow_x \right|_{\mathcal{F}_t} = \frac{R(\xi_t)}{R(x)} \ind{\tau > t}\cdot \left. \P_x \right|_{\mathcal{F}_t}. 
\end{equation}
The probability $\P^\uparrow_x$ is, in the sense of Doob's $h$-transform, the law of the Lévy process $\xi$ started from $x$ conditioned to stay positive.


Lévy processes conditioned to stay positive have been the subject of a large literature. We simply recall from \cite[Proposition~1]{ChD05} that, for all $x > 0$,  the process $\xi$ is transient under ${\P}^\uparrow_x$. Moreover, \cite[Theorem~1]{ChD05} gives the following path-decomposition at the last passage time of the overall minimum. 
\begin{lemma}[{\cite[Theorem~1]{ChD05}}]
	\label{lem:chd}
	We set
	$
	\underline{\xi} = \inf_{s \geq 0} \xi_s$ and $m = \sup\{ t \geq 0 : \min(\xi_t,\xi_{t-}) = \underline{\xi} \}.
	$
	For all $x>0$ and $0 \leq y \leq x $, we have
	\[
	\P^{\uparrow}_x (\underline{\xi} \ge y ) = \frac{R(x-y)}{R(x)} \ind{y\le x}.
	\]
	Moreover, the process $(\xi_{s+m}-\underline{\xi}, s \geq 0)$ is independent of $(\xi_s, s \leq m)$; the law of the former process, that we write $\P^\uparrow$, does not depend on $x$.
\end{lemma}

\subsection{Perpetual integrals of a Lévy process conditioned to stay positive}
\label{subsec:perpetualIntegral}

Given $\xi$ a càdlàg process on $\R_+$ and $f$ a measurable function $\R_+ \to \R_+$ a perpetual integral of $\xi$ is a variable defined as $\int_0^\infty f(\xi_s)\dd s$.
The main result of this section is an integral criterion on $f$ for the finiteness of perpetual integrals of the Lévy process conditioned to stay positive. 
\begin{proposition}
	\label{prop:finiteness}
	Let $\xi$ be a centred Lévy process with finite variance starting from $x>0$. We denote by $\P^\uparrow_x$ the law of this process conditioned to stay positive, as defined in \eqref{eqn:lawConditioned}. For every eventually non-increasing non-negative bounded function $f\colon [0,\infty) \to [0,\infty)$, we have
	\begin{equation}
	\label{eqn:finiteness}
	\begin{split}
	\int_0^\infty f(\xi_s) \dd s < \infty \quad \P^\uparrow_x\text{-a.s.} &\iff \int_0^\infty yf(y )\dd y  < \infty\\
	\int_0^\infty f(\xi_s) \dd s = \infty \quad \P^\uparrow_x\text{-a.s.} &\iff \int_0^\infty yf(y) \dd y = \infty.
	\end{split}
	\end{equation}
\end{proposition}

We mention that the assumptions on $f$ could be relaxed by adding requirements on the law of $\xi$, such as requiring the law of $\xi_1$ to be absolutely continuous with respect to the Lebesgue measure. We also observe that 
\begin{equation}
\label{eqn:finite<C}
\int_0^\infty f(\xi_s)\ind{\xi_s \leq C} \dd s < \infty,\qquad \P^\uparrow_x\text{-a.s for all}~ C > 0,
\end{equation}
 by \cite[Proposition~1]{ChD05} and boundedness of $f$, and that $\int_0^C y f(y) \dd y < \infty$ by boundedness of $f$. As a result, without loss of generality, we can assume the function $f$ in Proposition~\ref{prop:finiteness} to be non-increasing.

If $\xi$ is a Brownian motion under law $\P$, then $\xi$ under $\P^{\uparrow}_x$ is a 3-dimensional Bessel process. In that case, a similar result is known \cite[Exercise~XI~2.5]{ReY}. For a random walk conditioned to stay positive, the corresponding result is given by \cite[Proposition~2.1]{Che15}. For unconditioned L\'evy processes, the perpetual integrals have been studied in \cite{DoK16,KoS19}. Recent development of Baguley, Döring and Kyprianou~\cite{BDK20} gives a new characterization for transient strong Markov processes.
However, to directly apply their criterion would require characterizing all so-called \emph{supportive sets} of a  Markov process $X$, i.e.\ all sets $A$ such that $\P( \{X_t, t\ge 0\}\subseteq A ) >0$. 
It is unclear how to do this for conditioned L\'evy processes. 
For this reason, Proposition~\ref{prop:finiteness} is not an immediate consequence of \cite{BDK20} but we will develop a proof based on an intermediate result developed there (Lemma~\ref{lem:BDK} below).

We decompose the proof of Proposition~\ref{prop:finiteness}  into three lemmas. In Lemma~\ref{lem:DoB}, we show that $\int f(\xi_s) \dd s$ is finite $\P^\uparrow_x$-a.s.\ if and only if its mean is finite, using Lemma~\ref{lem:BDK}. We then show in Lemma~\ref{lem:integral} the equivalence between the finiteness of the mean and the integral test of the proposition. Finally, we show that $\int_0^\infty f(\xi_s) \dd s < \infty$ has a 0-1 law, which yields the equivalence between the two statements in \eqref{eqn:finiteness}.

\begin{lemma}
	\label{lem:DoB}
	For any bounded non-increasing measurable function $f\ge 0$ and $x \geq 0$, we have
	\[
	\P_x^{\uparrow}\left( \int_0^\infty f(\xi_s) \dd s < \infty \right)=1  \iff \exists C>0,~ \int_0^\infty \E_x^{\uparrow}(f(\xi_s)\ind{\xi_s\ge C}) \dd s <\infty.
	\]
\end{lemma}

The proof of Lemma~\ref{lem:DoB} is based on the following observation that we take from \cite{BDK20}. For unconditioned L\'evy processes, an analogous result is given in \cite[Lemma 4.5]{KoS19}.

\begin{lemma}[{\cite[Proposition 2.7]{BDK20}}]\label{lem:BDK}
	Let $\zeta$ be a transient strong Markov process on $\R_+$ and denote by $\mathbf{P}_x$ the law of $\zeta$ starting from $\zeta_0 =x$. For any non-negative measurable function $f$ and $N, p>0$, let 
	\[
	M_{N,p} := \left\{y\in \R_+\colon \mathbf{P}_y \left(\int_0^{\infty} f (\zeta_s) \dd s <N\right) > p \right\}. 
	\]
	Suppose that $x\in M_{N,p}$, then 
	\[
	\int_0^\infty \mathbf{E}_x(f(\zeta_s)\ind{\zeta_s\in M_{N,p}}) \dd s <2N/p^2. 
	\]
\end{lemma}

\begin{proof}[Proof of Lemma~\ref{lem:DoB}]
	We first observe that if there exists $C > 0$ such that
	\begin{equation*}\label{eqn:finite>C}
	\int_0^\infty \E_x^{\uparrow}(f(\xi_s)\ind{\xi_s\ge C}) \dd s = \E_x^\uparrow \left( \int_0^\infty f(\xi_s) \ind{\xi_s \geq C} \dd s \right) <\infty,
	\end{equation*}
	then $\int_0^\infty f(\xi_s) \ind{\xi_s \geq C} \dd s< \infty$ $\P^\uparrow_x$-a.s.. Combining this and \eqref{eqn:finite<C} completes the proof of the reverse part of the lemma.
	
	We now turn to the direct part and assume that there exists $x \geq 0$ such that $\int_0^\infty f(\xi_s) \dd s < \infty$ $\P^\uparrow_x$-a.s.. 
	Then, given $p \in (0,1)$, we can choose $N$ large enough such that
	\[
	\P^{\uparrow}_x \left(\int_0^{\infty}\! f(\xi_s) \dd s \ge N\right)  < 1-p, 
	\]
	i.e., with $M_{N,p}$ defined as in Lemma~\ref{lem:BDK}, we have $x\in M_{N,p}$. 
	it remains to prove that there exists $C<\infty$ such that  $[C,\infty) \subset M_{N,p}$. Then the desired statement follows from Lemma~\ref{lem:BDK}.
	
	Set
	\[T_N= \inf \left\{t\colon \int_0^t f(\xi_s) \dd s \ge N \right\}.\]
	We first show that $\P^\uparrow_y(T_N < \infty)$ is comparable to $\inf_{0 \leq z \leq \delta y} \P^\uparrow_{y-z}(T_N < \infty)$ for large values of $y$. 
	Setting $\underline{\xi} = \inf_{s \geq 0} \xi_s$, we observe that, for all $y,z \geq 0$,
	\begin{align*}
	\P^\uparrow_{y+z}(T_N < \infty) &\leq \P^\uparrow_{y+z}(T_N < \infty, \underline{\xi} \geq z) + \P^\uparrow_{y+z}(\underline{\xi} < z)
	\leq \P^\uparrow_{y+z}\left(T_N < \infty \,\middle|\,\underline{\xi} \geq z\right) + \P^\uparrow_{y+z}(\underline{\xi} < z).
	\end{align*}
	Moreover, the law of $\xi$ under $\P^\uparrow_{y+z}(\,\cdot\,\mid\underline{\xi} \geq z)$ is the same as the law of $\xi + z$ under $\P^\uparrow_{y}(\,\cdot\,\mid\underline{\xi} \geq 0) = \P^\uparrow_{y}$. It follows that 
	\[
	\P^\uparrow_{y+z}\left(T_N < \infty \,\middle|\,\underline{\xi} \geq z\right) = \P^\uparrow_y\left( \int_0^\infty f(\xi_s + z) \dd s < N  \right) \leq \P^\uparrow_y(T_N < \infty),
	\]
	as $f$ is non-increasing. Therefore, we have 
	\[
	\P^\uparrow_y(T_N < \infty) \geq \P^\uparrow_{y+z}(T_N < \infty) - \P^\uparrow_{y+z}(\underline{\xi} < z),
	\]
	where $\P^\uparrow_{y+z}(\underline{\xi} < z) =  1 - \frac{R(y)}{R(y+z)}$ by Lemma~\ref{lem:chd}.
	
	Let $\epsilon \in (0,1)$ and $\delta < \epsilon/2$. By \eqref{eqn:estimatesRenewal}, there exists $C_1(\epsilon)>0$ such that for all $y>C_1(\epsilon)$ and $z \in [0,\delta y]$ we have $\frac{R(y)}{R(y+z)} = \P^\uparrow_{y+z}(\underline{\xi} \geq z) > 1-\epsilon$. As a result, for all $y \geq C_1(\epsilon)$, we have
	\begin{equation}\label{eqn:delta-y}
	\inf_{0 \leq z \leq \delta y} \P^\uparrow_{y-z}(T_N < \infty) \geq \P^\uparrow_y(T_N < \infty) - \epsilon. 
	\end{equation}
	
	We now set $\hat{T}_{y,\delta} = \inf\{ s \geq 0 : \xi_s \in [y(1-\delta),y] \}$. As $\xi$ is a Lévy process with finite variance under law $\P$, it is well-known that the overshoot distribution of $\xi$ is tight under law $\P_x^\uparrow$. Indeed, by \cite[Lemma~3]{BeS11} this holds under $\P$. Therefore, it also holds under $\P_x^\uparrow$ due to the duality property \cite[Corollary~2]{BeS11}. Therefore, we have
	\[
	\P^\uparrow_x(\hat{T}_{y,\delta} < \infty) \geq \P^\uparrow_x\left( \xi_{\inf\{t > 0 : \xi_t >y(1-\delta)\}} - y (1-\delta) \leq \delta y \right)  \underset{y\to \infty}{\longrightarrow} 1. 
	\]
	Then we set $C_2(\epsilon)$ such that $\P^\uparrow_x(\hat{T}_{y,\delta} = \infty) < \epsilon$ for all $y \geq C_2(\epsilon)$.
	
	Recall that $x\in M_{N,p}$. For all $y \geq C(\epsilon):= \max(C_1(\epsilon),C_2(\epsilon))$, using the strong Markov property and \eqref{eqn:delta-y}, we deduce that 
	\begin{align*}
	\P^{\uparrow}_x \left(\int_0^{\infty}\! f(\xi_s) \dd s \ge N\right)
	&\ge \E^{\uparrow}_x \left(\ind{\hat{T}_{y,\delta}<\infty}\P^{\uparrow}_{\xi_{\hat{T}_{y,\delta}} } \left(\int_0^{\infty}\! f(\xi_s) \dd s \ge N\right)  \right) \\
	&\ge \E^{\uparrow}_x \left( \ind{\hat{T}_{y,\delta}<\infty}\P^{\uparrow}_{y} \left(\int_0^{\infty}\! f(\xi_s) \dd s \ge N\right) \right) -\epsilon \\
	&\ge \P^{\uparrow}_{y} \left(\int_0^{\infty}\! f(\xi_s) \dd s \ge N\right)  (1-\epsilon) -\epsilon. 
	\end{align*}
	By choosing $\epsilon \in (0,1)$ small enough, we have 
	\[
	\P^{\uparrow}_y \left(\int_0^{\infty}\! f(\xi_s) \dd s \ge N\right) < \frac{\P^{\uparrow}_x \left(\int_0^{\infty}\! f(\xi_s) \dd s \ge N\right) + \epsilon}{1-\epsilon} < 1-p.
	\]
	So we have $[C(\epsilon),\infty) \subset M_{N,p}$, completing the proof. 
\end{proof}

\begin{lemma}
	\label{lem:integral}
	Under the same assumptions as in Proposition~\ref{prop:finiteness}, for any $x \geq 0$, 
	there exists $0<c_1 < c_2 < \infty$ such that
	\[
	c_1 \int_x^\infty y f(y) \dd y \leq \int_0^\infty \E^\uparrow_x(f(\xi_s)) \dd s \leq c_2 \int_0^\infty y f(y) \dd y.
	\] 
\end{lemma}
Note that, by applying Lemma~\ref{lem:integral} with $\bar{f}(x)= f(x)\ind{x\ge C}$ and using the boundedness of $f$, we 
deduce that, under the same assumptions,  
\[
\exists C>0,~ \int_0^\infty \E_x^{\uparrow}\left(f(\xi_s)\ind{\xi_s\ge C}\right) \dd s <\infty\iff
\int_0^\infty yf(y) dy < \infty.
\]
\begin{proof}
	We first assume that $\xi$ is not a compound Poisson process. It then follows from the change of measure \eqref{eqn:lawConditioned} and \cite[Theorem~VI.20]{Ber96} that 
	\begin{align*}
	\int_0^\infty \E^\uparrow_x \left( f(\xi_t) \right) \dd t
	&= \frac{1}{R(x)}\int_0^\infty \E_x \left( R(\xi_t) f(\xi_t) \ind{t< \tau} \right) \dd t\\
	&= C' \int_{[0,\infty)}  \dd \bar{R}(r) \int_{[0,x]} \dd R(z) R(x+r-z)f(x+r-z), 
	\end{align*}
	where $C'>0$ is a certain constant, $\tau=\inf\{s \geq 0\colon \xi_s <0\}$ and $\bar{R}$ stands for the renewal function of $-\xi$. Additionally, since the unconditioned L\'evy process $\xi$ is centred with finite variance, by \cite[Theorem~7]{DoM02} and \cite[Theorem~I.21]{Ber96}, the measures $\dd R(\cdot + z)$ and $\dd \bar{R}(\cdot + z)$ converge vaguely toward multiples of the Lebesgue measures, as $z \to \infty$. Using \eqref{eqn:estimatesRenewal} as well, we deduce that there exist two constants $0 < c < C< \infty$ such that
	\begin{multline*}
	c \int_{[0,x]} \dd z \int_{[x-z,\infty)} y f(y)  \dd y  
	\le \int_{[0,\infty)}  \dd \bar{R}(r) \int_{[0,x]} \dd R(z) R(x+r-z)f(x+r-z) \\
	\le C \int_{[0,x]} \dd z \int_{[x-z,\infty)} yf(y)  \dd y  . 
	\end{multline*}
	As the function $f$ is bounded, this leads to the desired statement. 
	
	If $\xi$ is a compound Poisson process, then the corresponding result on random walks leads to the conclusion. We recall that the span of the Lévy process $\xi$ is defined as $r:=\sup\{s > 0 : \P(\xi_1 \not \in s \Z) = 0 \}\ge 0$, with the convention that $\sup \emptyset = 0$. Assuming that $\xi$ is non-lattice (i.e.\@ the span is $r=0$), it is a consequence of \cite[Equation (2.9)]{Che15} and estimates on the renewal functions of random walks that can be found in \cite[Chapter~XII]{Fel}. If $\xi$ is lattice with span $r>0$, a similar argument leads to 
	\[
	\int_0^\infty \E^\uparrow_x \left( f(\xi_t) \right) \dd t<\infty \iff \sum_{k = 0}^\infty kf(kr+x)<\infty.
	\]
	As $f$ is eventually non-increasing, there exists $0<c < C < \infty$ such that
	\[
	c \int_x^\infty y f(y) \dd y \leq \sum_{k=0}^\infty k f(kr+x) \leq C \int_0^\infty y f(y) \dd y,
	\]
	completing the proof.
\end{proof}

We now prove that the finiteness of a perpetual integral of a Lévy process conditioned to stay positive satisfies a zero-one law.
\begin{lemma}
	\label{lem:01}
	Under the assumptions of Proposition~\ref{prop:finiteness}, for all $x \geq 0$ we have
	\[\P^\uparrow_x\left( \int_0^\infty f(\xi_s) \dd s < \infty \right) \in \{ 0,1\}.\]
\end{lemma}

\begin{proof}
	Since $\xi$ is transient under $\P^{\uparrow}_x$ and $f$ is bounded,   
	we may assume without loss of generality that $f$ is non-increasing on the entire half-line $[0,\infty)$.

	We introduce the function
	\[
	\psi : x \in \R_+ \mapsto \P^\uparrow_x \left( \int_0^\infty f(\xi_s) \dd s <\infty \right).
	\]
	Note that $\psi$ is measurable, non-negative and bounded, by standard properties of Markov processes. 
	
	We first claim that $(\psi(\xi_t), t \geq 0)$ is a closed martingale.
	Indeed, for all $x,t \geq 0$, as $\int_0^t f(\xi_s) \dd s < \infty$ $\P^\uparrow_x$-a.s., we have
	\begin{align*}
	\P^\uparrow_x \left( \int_0^\infty f(\xi_s) \dd s <\infty \,\middle|\, \mathcal{F}_t\right)
	&= \P^\uparrow_x \left( \int_0^\infty f(\xi_{t+s}) \dd s < \infty \,\middle|\, \mathcal{F}_t\right)\\
	&= \P^\uparrow_{\xi_t} \left( \int_0^\infty f(\xi_s) \dd s <\infty\right) = \psi(\xi_t), \quad \P^\uparrow_x\text{-a.s..}
	\end{align*}
	Therefore, $(\psi(\xi_t), t \geq 0)$ is a closed martingale. In particular, this yields
	\begin{equation}
	\label{eqn:limite}
	\lim_{t \to \infty} \psi(\xi_t) = \ind{\int_0^\infty f(\xi_s) \dd s <\infty}, \quad \P^\uparrow_x - \text{a.s..}
	\end{equation}
	
	We next prove that $\psi$ is non-increasing in $x$, using Lemma~\ref{lem:chd}. With notation therein, we observe that, for all $x>0$, $\int_0^m f(\xi_s) \dd s \!<\! \infty$ $\P^\uparrow_x$-a.s.. It follows that
	\begin{align*}
	\psi(x)
	= \P^\uparrow_x \left( \int_m^\infty f(\xi_s) \dd s <\infty \right)
	&= \P^\uparrow_x \left( \int_0^\infty f(\xi_{s+m}-\underline{\xi}+\underline{\xi}) \dd s < \infty \right)\\
	&= \P^\uparrow \left( \int_0^\infty f(\xi_s + v_x) \dd s < \infty \right),
	\end{align*}
    where $v_x$ is, under law $\P^\uparrow$, an independent variable of $\xi$ with the same law as the variable $\underline{\xi}$ under law $\P_x^\uparrow$. For all $x < y$, $v_x$ is stochastically dominated by $v_y$. Hence, as $f$ is non-increasing,
	$\int_0^\infty f(\xi_s + v_x) \dd s$ is stochastically larger than $\int_0^\infty f(\xi_s + v_y) \dd s$. 
	It follows that $\psi$ is non-increasing and non-negative, and hence $\lim_{x \to \infty} \psi(x) = \lambda$ exists. 
	
	Since $\xi_t \to \infty$ $\P^\uparrow_x$-a.s.\ by transience, we deduce by \eqref{eqn:limite}  that 
	\[
	\lambda =\lim_{t \to \infty} \psi(\xi_t)= \ind{\int_0^\infty f(\xi_s) \dd s <\infty},\quad \P_x^\uparrow\text{-a.s..}
	\]
	This implies that $\lambda \in \{0,1\}$ and that $\int_0^\infty f(\xi_s) \dd s <\infty$ holds with probability $0$ or $1$, depending on the value of $\lambda$. This completes the proof.
\end{proof}

\begin{proof}[Proof of Proposition~\ref{prop:finiteness}]
	By Lemmas~\ref{lem:DoB} and~\ref{lem:integral}, we first observe that
	\[
	\int_0^\infty f(\xi_s) \dd s < \infty \quad \P^\uparrow_x\text{-a.s.} \iff \int_0^\infty yf(y) \dd y < \infty.
	\]
	Then, as $\P^\uparrow_x\big(\int_0^\infty f(\xi_s) \dd s < \infty\big) \in \{0,1\}$ by Lemma~\ref{lem:01}, we deduce by contraposition that 
	\begin{align*}
	\int_0^\infty yf(y) \dd y = \infty &\iff \P^\uparrow_x\left(\int_0^\infty f(\xi_s) \dd s < \infty\right) < 1
	\iff \P^\uparrow_x\left(\int_0^\infty f(\xi_s) \dd s < \infty\right) =0,
	\end{align*}
	completing the proof.
\end{proof}

\section{Truncated derivative martingales and the spinal decomposition}
\label{sec:truncatedMartingale}

In this section, we use the renewal function of a Lévy process to introduce the truncated versions of a derivative martingale. We show that the non-degeneracy of the limit of the derivative martingale is equivalent to the uniform integrability of the truncated derivative martingales. We then give the spinal decomposition, describing the law of a branching Lévy process biased by the truncated martingales.

\subsection{Truncated derivative martingales}
\label{subsec:truncatedMartingale}


\begin{lemma}
	\label{lem:truncatedDerivativeMartingale}
	Let $b>0$. We set for $t \geq 0$:
	\begin{equation}
	\label{eqn:defTruncatedDerivativeMartingale}
	Z^b_t = \sum_{u \in \mathcal{N}_t} R(X_u(t)+b) \ind{\inf_{s \leq t} X_u(s) \geq -b} e^{-X_u(t)}.
	\end{equation}
	The process $Z^b:=(Z^b_t, t \geq 0)$ is a non-negative martingale, called the \emph{truncated derivative martingale}, that converges a.s.\ to a limit $Z^b_\infty\ge 0$ as $t\to\infty$.
\end{lemma}

\begin{proof}
	The fact that $Z^b$ is a non-negative martingale is a straightforward consequence of the branching property of the branching Lévy process, the many-to-one lemma and equation \eqref{eqn:formulaRenewal}.
\end{proof}

The family $(Z^b,b \geq 0)$ approaches the derivative martingale $Z$ in the following sense. 
\begin{lemma}
	\label{lem:asymptotTruncated}
	Under assumptions \eqref{eqn:H0} and \eqref{eqn:var}, $Z_t$ converges $\P$-a.s.\@ to a non-negative 
	limit $Z_\infty\ge 0$ as $t\to \infty$. Moreover, with $c_{\star}>0$ the constant given in \eqref{eqn:estimatesRenewal}, there is the identity
	\[
	c_\star Z_\infty = \lim_{b \to \infty} Z^b_\infty, \quad \P\text{-a.s..}
	\]
\end{lemma}

\begin{proof}
	By \eqref{eqn:estimatesRenewal}, the function $R$ satisfies $R(x) \sim c_\star x$ as $x \to \infty$. Hence, for all $\epsilon > 0$, there exists $A_{\epsilon} > 0$ such that $x \in [R(x)/(c_\star + \epsilon),R(x)/(c_\star - \epsilon) ]$ for all $x \geq A_{\epsilon}$. In particular, for every $t\ge 0$,
	\[
	\frac{1}{c_\star + \epsilon} Z^b_t \leq  Z_t + b W_t \leq \frac{1}{c_\star-\epsilon} Z^b_t, \quad \text{on the event}\quad \Big\{\inf_{s \leq t} M_s \geq A_{\epsilon} - b\Big\},
	\]
where $W_t = \sum_{u \in \mathcal{N}_t} e^{-X_t(u)}$ is the additive martingale and $M_s = \inf_{u \in \mathcal{N}_s} X_s(u)$. 
	
	By \cite[Theorem~1.1]{BeM18b}, under assumption \eqref{eqn:H0} we have $\lim_{t \to \infty} W_t = 0$ a.s.. Since $e^{-M_t} \leq W_t \to 0$, it follows that $\inf_{t \geq 0} M_t > -\infty$ a.s..
	Therefore, letting $t \to \infty$, we have 
	\[
	\frac{1}{c_\star + \epsilon} Z^b_\infty\le \liminf_{t\to \infty}Z_t\le  \limsup_{t \to \infty} Z_t \le \frac{1}{c_\star-\epsilon} Z^b_\infty\quad \text{on the event} \quad \Big\{\inf_{t \geq 0} M_t \geq A_{\epsilon} - b\Big\} . 
	\]
	Observe that $b \mapsto Z_t^b$ is non-decreasing  for all $t \geq 0$, so $b \mapsto Z_\infty^b$ is a.s.\ non-decreasing. Therefore, $\lim_{b \to \infty} Z_\infty^b$ exists a.s.. Then letting $b \to \infty$, as $\lim_{b \to \infty}\P(\inf_{t \geq 0} M_t \geq A_{\epsilon} - b)=1$, we deduce that 
	\[
	\lim_{b\to\infty} \frac{1}{c_\star + \epsilon} Z^b_\infty\le \liminf_{t\to \infty}Z_t\le  \limsup_{t \to \infty} Z_t \le \frac{1}{c_\star-\epsilon} \lim_{b\to\infty} Z^b_\infty, \quad \P\text{-a.s.}.
	\]
	Finally, letting $\epsilon \to 0$ leads to the desired statement. 
\end{proof}

The previous lemma allows us to study the non-degeneracy of the limit of the derivative martingale $Z$ via the uniform integrability of the truncated martingales $Z^b$. 

\begin{corollary}
	\label{cor:ncandsc}
	If there exists $b > 0$ such that $Z^b$ is uniformly integrable, then $Z_\infty$ is non-degenerate. 
	On the other hand, if $Z^b_\infty = 0$ a.s.\ for all $b > 0$, then $Z_\infty = 0$ a.s..
\end{corollary}

\begin{proof}
	Assume first there exists $b>0$ such that $Z^b$ is uniformly integrable, then $\P(Z^b_\infty>0)>0$.  As $b \mapsto Z^b_\infty$ is non-decreasing, we deduce by Lemma~\ref{lem:asymptotTruncated} that $\P(Z_\infty > 0) > 0$. 
	
	If $Z^b_\infty=0$ a.s.\ for every $b>0$, then $Z_\infty=0$ a.s.\ by Lemma~\ref{lem:asymptotTruncated}.
\end{proof}

\begin{remark}
	\label{rem:ncs}
	In proving Theorem~\ref{thm:main}, we will also show that the following three facts are equivalent:
	\begin{enumerate}
		\item $Z_\infty > 0$ a.s.\ on the survival set of the branching Lévy process;
		\item there exists $b > 0$ such that $Z^b$ is uniformly integrable;
		\item for every $b > 0$, $Z^b$ is uniformly integrable.
	\end{enumerate}
\end{remark}

\subsection{Spinal decompositions of the branching Lévy process}
\label{sec:spinalDecompositionProper}

The spinal decomposition consists in an alternative description of the law of a branching process biased by a non-negative martingale as a branching process with a special individual called the spine. This alternative description in turns allows us to study whether the martingale is uniformly integrable or not, thanks to the following classical argument.
\begin{fact}
	\label{fct}
	Let $(M, \mathcal{F})$ be a non-negative $\P$-martingale with $\E(M_0)=1$ and set $\Q := M \cdot \P$ to be the law $\P$ biased by the martingale $M$, which means that, for each $t \geq 0$, 
	$
	\left. \frac{\dd \Q}{\dd \P} \right|_{\mathcal{F}_t} = M_t$ a.s..
	 We have
	\begin{align*}
	(M_t, t \geq 0) \text{ is uniformly integrable} &\iff \liminf_{t \to \infty} M_t < \infty \quad \Q\text{-a.s..}\\
	M_\infty = 0 \quad \P\text{-a.s.} &\iff \limsup_{t \to \infty} M_t = \infty \quad \Q\text{-a.s..}
	\end{align*}
\end{fact}

Fact~\ref{fct} is a consequence of \cite[Theorem 5.3.3]{Dur}, and the fact that $(1/M_t, t \geq 0)$ is a non-negative $\Q$-supermartingale, thus having a finite limit $\Q$-a.s..  
In view of Fact~\ref{fct}, we will study the spinal decomposition associated with a  truncated martingale $Z^b$ and explore its asymptotic behaviour.

To this end, we begin by introducing the spinal decomposition associated to the critical additive martingale $(W_t= \sum_{u \in \mathcal{N}_t} e^{-X_t(u)}, t\ge 0)$. 
Let $\P$ be the law of the branching Lévy process and $\mathcal{F}$ its natural filtration. Using the $\P$-martingale $(W,\mathcal{F})$, we define the measure $\bar{\P}$ by
\begin{equation}
\label{eqn:barPdef}
\bar{\P} := W \cdot \P. 
\end{equation}
This change of measure was considered in \cite{BeM18b,IkM18} to study the asymptotic behaviour of additive martingales, and is the counterpart of results of Lyons \cite{Lyo97} for branching random walks.

To obtain an alternative representation of $\bar{\P}$, we construct a new branching process, with a distinguished individual called the spine. Specifically, define a sigma-finite measure $\hat{\Lambda}$ on $\mathcal{P} \times \N$ by
\begin{equation}
\label{eqn:hatLambdaDef}
\hat{\Lambda}(\dd \x, \dd k) := \sum_{j \geq 1} e^{-x_j} \Lambda(\dd \x) \delta_j(\dd k).
\end{equation}
Let $\beta$ be a Brownian motion and $\hat{{N}}(\dd t, \dd \x, \dd k)$ an independent Poisson random measure on $\R_+ \!\times \!\mathcal{P} \!\times\! \N$ with intensity $\dd t\otimes \hat{\Lambda}(\dd \x,\dd k)$. Define a Lévy process $\hat{\xi}$ by the following compensated Poisson integral
\begin{multline}
\hat{\xi}_t = \sigma \beta_t + \hat{a} t + \int_{[0,t]\times \mathcal{P} \times \N} x_{k} \ind{|x_{k}| < 1} \hat{N}^c (\dd t, \dd \x, \dd k) \\
+ \int_{[0,t] \times \mathcal{P} \times \N} x_{k} \ind{|x_{k}| \geq 1} \hat{N} (\dd t, \dd \x, \dd k), 
\end{multline}
with $\hat{a}$ the quantity defined in \eqref{eqn:defhata}. The process $\hat{\xi}$ is well-defined and finite, thanks to equations \eqref{eqn:levyEve} and \eqref{eqn:expIntegrability} (see \cite{BeM18b} for more details on this construction). Moreover, $\hat{\xi}$ is a Lévy process with Lévy-Khinchine exponent $\Psi$ defined in \eqref{eqn:defPsi}.

The branching Lévy process with spine is then constructed as follows. The spine particle follows the trajectory of $\hat{\xi}$, while making offspring according to the point process $\hat{{N}}$. More precisely, for all atoms $(t,\x,k)$ of $\hat{N}$, the spine particle jumps at time $t$ from position $\hat{\xi}_{t-}$ to $\hat{\xi}_t = \hat{\xi}_{t-} + x_{k}$, while for all $j \neq k$, it creates a new particle at position $\hat{\xi}_{t-} + x_j$. Each newborn particle then starts from its current birth time and location a new independent branching Lévy process with law $\P$. The set of particles alive at time $t$ is again denoted by $\mathcal{N}_t$. For $u\in \mathcal{N}_t$, let $(X_s(u), s \leq t)$ be the trajectory of this particle. The label of the spine particle at time $t$ is written as $w_t \in \mathcal{N}_t$. The law of the branching Lévy process with spine $(X,\mathcal{N},w)$ thus defined is denoted by $\hat{\P}$. The spinal decomposition is the following result.

\begin{lemma}
	\label{lem:spinalDecomposition}
	The law of $(X,\mathcal{N})$ is the same under laws $\bar{\P}$ and $\hat{\P}$. Moreover, one has
	\begin{equation}
	\label{eqn:chooseSpine}
	\hat{\P}(w_t = u \mid \mathcal{F}_t) = e^{-X_t(u)} / W_t, \quad \forall t\ge 0. 
	\end{equation}
\end{lemma}

We refer to \cite[Theorem 5.2]{ShW19} for the proof of the spinal decomposition for branching Lévy processes, and to \cite[Lemma 2.3]{BeM18b} for a simple argument based on the spinal decomposition of branching random walks which could be adapted to our settings. The spinal decomposition was introduced by Lyons, Pemantle and Peres in \cite{LPP95} for Galton-Watson processes. The result was then refined by Lyons \cite{Lyo97} to study additive martingales in a branching random walk, and was further extended to general martingales based on additive functionals of a branching random walk in \cite{BiK04}.

We now discuss a similar extension in the settings of branching Lévy processes. Consider for every $b>0$ the law $\Q^b$ defined by
\begin{equation*}
\Q^b = \frac{Z^b}{R(b)} \cdot \P.
\end{equation*}
Thanks to Lemma~\ref{lem:spinalDecomposition}, one straightforwardly obtains a spinal decomposition result for the law $\Q^b$.

\begin{lemma}
	\label{lem:spinalDerivative}
	Let $b > 0$, we define a measure $\hat{\Q}^b$ by setting for all $t \geq 0$
	\begin{equation}
	\left.\frac{\dd \hat{\Q}^b}{\dd \hat{\P}}\right|_{\mathcal{F}_t} = \frac{R(X_t(w_t) + b)}{R(b)} \ind{\inf_{s \leq t} \!X_s(w_t) > -b}.
	\end{equation}
	Then the law of $(X,\mathcal{N})$ is the same under laws $\Q^b$ and $\hat{\Q}^b$.
\end{lemma}

\begin{proof}
	Let $V$ be an $\mathcal{F}_t$-measurable random variable. We observe that
	\begin{align*}
	\E_{\hat{\Q}^b} \left( V \right) &= \hat{\E}\left( V \frac{R(X_t(w_t) + b)}{R(b)} \ind{\inf_{s \leq t} X_s(w_t) > -b} \right)\\
	&= \hat{\E}\bigg( \frac{V}{W_t} \sum_{u \in \mathcal{N}_t} e^{-X_t(u)} \frac{R(X_t(u) + b)}{R(b)} \ind{\inf_{s \leq t} X_s(u) > -b} \bigg),
	\end{align*}
	by conditioning on $(X,\mathcal{N})$ and using Lemma~\ref{lem:spinalDecomposition}. Then, as $\{W_t = 0\} \subset \{ Z^b_t = 0\}$, one has
	\[
	\E_{\hat{\Q}^b} \left( V \right) = \bar{\E}\bigg( \frac{Z^b_t}{R(b) W_t} V \bigg) = \E_{\Q^b} (V). \qedhere
	\]
\end{proof}

In light of Lemma~\ref{lem:spinalDerivative} and the spinal construction of $\hat{\P}$, we can still describe $\hat{\Q}^b$ as the law of the particle system in which the spine particle follows the trajectory $\hat{\xi}$, while making offspring according to the point process $\hat{{N}}$. 
However, we stress that under law $\hat{\Q}^b$, $\hat{\xi}$ is a Lévy process with characteristic exponent $\Psi$ conditioned to stay above level $-b$, in the sense of Section~\ref{sec:levy-cond}; in other words, 
\begin{equation}\label{eqn:SpineTruncated}
\hat{\xi}~\text{under}~ \hat{\Q}^b~\text{has the same distribution as}~ (\xi - b)~\text{under}~ \P^{\uparrow}_b.
\end{equation}
In the rest of the article, we apply the results of Section~\ref{sec:levy-cond} to the process $\hat{\xi}+b$ under law $\hat{\Q}^b$, which is a Lévy process conditioned to stay positive.

Under $\hat{\Q}^b$, $\hat{N}$ is also no longer a Poisson random measure with intensity $\dd t\otimes \hat{\Lambda}(\dd \x,\dd k)$. We can nevertheless compute its compensator in the following sense.
\begin{lemma}\label{lem:2}
For any non-negative predictable function $W$ on $\R_+\times \mathcal{P} \times \N$, we have 
	\begin{align*}
	&	\E_{\hat{\Q}^b} \left(\int_{\R_+\times \mathcal{P}\times \N} W(t,\x ,k) \dd \hat{{N}}(\dd t, \dd \x, \dd k)\right)\\
	&= \E_{\hat{\Q}^b} \left(\int_{\R_+\times \mathcal{P}\times \N} W(t,\x ,k)\ind{b+\hat{\xi}_{t-}+x_k >0} \frac{R(b+ \hat{\xi}_{t-}+x_k)}{R(b+\hat{\xi}_{t-})}  \dd t \otimes \hat{\Lambda}(\dd \x,\dd k)\right).
	\end{align*}
\end{lemma}

\begin{proof}
	Write $Z_t:= \frac{R(\hat{\xi}_t + b)}{R(b)} \ind{\inf_{s \leq t} \!\hat{\xi}_s > -b}$, which is the density process for the change of measure in Lemma~\ref{lem:spinalDerivative}.  
	It is well-known the predictable compensator (see e.g.\ \cite[Section II.1]{JaS03}) of 
	the Poisson random measure of $\hat{N}$ under $\hat{\P}$ is given by its intensity $\dd t \otimes \hat{\Lambda}(\dd \x,\dd k)$.  Since we have that, for any non-negative predictable function $W$ on $\R_+\times \mathcal{P} \times \N$, 
	\begin{align*}
		&	\hat{\E} \left(\int_{\R_+\times \mathcal{P}\times \N} Z_t W(t,\x ,k)  \hat{{N}}(\dd t, \dd \x, \dd k)\right)\\
		&= \hat{\E} \left(\int_{\R_+\times \mathcal{P}\times \N}  Z_{t-}\ind{b+\hat{\xi}_{t-}+x_k >0} \frac{R(b+ \hat{\xi}_{t-}+x_k)}{R(b+\hat{\xi}_{t-})} W(t,\x ,k) \hat{{N}}(\dd t, \dd \x, \dd k)\right), 
	\end{align*}
and that $(t,\x,k)\mapsto \ind{b+\hat{\xi}_{t-}+x_k >0} \frac{R(b+ \hat{\xi}_{t-}+x_k)}{R(b+\hat{\xi}_{t-})} $ is predictable, 
 the result follows from Girsanov's theorem for random measures (see e.g.\ \cite[Theorem III.3.17(b)]{JaS03}).
\end{proof}

\section{Proof of the main result}
\label{sec:proof}

Using Fact~\ref{fct}, Lemma~\ref{lem:spinalDerivative} and Corollary~\ref{cor:ncandsc}, we now prove Theorem~\ref{thm:main}. We begin by giving alternative expressions of the condition \eqref{H1} that will help us use Proposition~\ref{prop:finiteness}. We then show that \eqref{H1} implies the non-degeneracy of the limit of the derivative martingale and  finally that the derivative martingale converges to $0$ if \eqref{H1} does not hold.

\subsection{Equivalent integral tests for \texorpdfstring{\eqref{H1}}{H1}}

We give here some equivalent formulas for \eqref{H1}. Recall that, for all $\x \in \mathcal{P}$,
$
Y(\x) = \sum_{k \geq 1} e^{-x_k}$ and $\tilde{Y}(\x) = \sum_{k \geq 1} x_k \ind{x_k \geq 0} e^{-x_k}.
$
Introduce the quantity
\[
\bar{Y}(\x) = \sum_{k \geq 1} ( 1 + x_k \ind{x_k \geq 0} ) e^{-x_k} = Y(\x) + \tilde{Y}(\x).
\]
We first rewrite \eqref{H1} in terms of the variable $\bar{Y}(\x)$.

\begin{lemma}
	\label{lem:H1bis}
	Under \eqref{eqn:H0} and \eqref{eqn:var}, \eqref{H1} is equivalent to
	\[
	\int_{\calP} Y(\x) \log_+ (\bar{Y}(\x) - 1)^2 \Lambda(\dd \x)< \infty.
	\]
\end{lemma}

\begin{proof}
	Let $B = \{ \x \in \mathcal{P} : \bar{Y}(\x)\! \geq\! 2\}$. 
	Since $\{\x \in \mathcal{P} :Y(\x)\!\ge\! 2 \text{ or } \tilde{Y}(\x)\!\ge\! 2\}\subset B$, we have
	\begin{align*}
	0 & = \int_{B^c} Y(\x) \log_+(Y(\x)-1)^2 \Lambda(\dd \x) = \int_{B^c} \tilde{Y}(\x) \log_+(\tilde{Y}(\x) - 1) \Lambda(\dd \x).
	\end{align*}
	So it suffices to prove that
	\begin{align*}
	&\int_B Y(\x) \log_+ (\bar{Y}(\x) - 1)^2 \Lambda(\dd \x)< \infty \\
	 \iff
	&\int_{B} Y(\x) \log_+(Y(\x)-1)^2 + \tilde{Y}(\x) \log_+(\tilde{Y}(\x) - 1) \Lambda(\dd \x) < \infty.
	\end{align*}
	Additionally, we have 
	\begin{equation}\label{eqn:Ybar>2}
	\Lambda(B) = \Lambda \left(\x \in \mathcal{P} : \bar{Y}(\x)\! \geq\! 2\right) <\infty. 
	\end{equation}
	Indeed, as $\sup_{x \in \R} x e^{-x} \ind{x \geq 0} \le e^{-1}$, we deduce that 
	$\bar{Y}(\x) \le \bar{Y}_2(\x)+ 1+e^{-1}$ for all $\x \in \mathcal{P}$, where we have set $\bar{Y}_2(\x):= \sum_{k \geq 2} ( 1 + x_k \ind{x_k \geq 0} ) e^{-x_k}$. 
	So we have 
	$\{\bar{Y}(\x)\! \geq\! 2\}\subset \left\{\bar{Y}_2(\x)\ge 1 - e^{-1}\right\}$.
	We also note that 
	\begin{equation}\label{eqn:Ybar>2-bis}
	\int_{\mathcal{P}} \bar{Y}_2(\x) \Lambda (\dd \x) 
	\le  \int_{\mathcal{P}} \sum_{k \geq 2} ( 2 + x^2_k ) e^{-x_k} \Lambda (\dd \x) <\infty,  
	\end{equation}
	where the finiteness of the last integral follows from 
	\eqref{eqn:varbis}. 
	As a result, we have 
	\begin{align}
	\int_{\mathcal{P}} \bar{Y}(\x)\ind{\bar{Y}(\x) \geq 2} \Lambda (\dd \x)
	&\le \int_{\mathcal{P}} (\bar{Y}_2(\x)+ 1+e^{-1})\ind{\bar{Y}_2(\x)\ge 1 - e^{-1}} \Lambda (\dd \x)  \nonumber\\
	&\le C \int_{\mathcal{P}}\bar{Y}_2(\x) \Lambda (\dd \x)<\infty. \label{eqn:Ybar>2-E}
	\end{align}
	This implies \eqref{eqn:Ybar>2}. 	
	Therefore, we may assume, without loss of generality, that $\Lambda$ is a probability distribution on $\mathcal{P}$. But in that case, the equivalence is a direct consequence of \cite[Lemma B.1]{Aid13} (for the reverse part) and \cite[Lemma A.1]{Mal} (for the direct part).
\end{proof}

Using Lemma~\ref{lem:H1bis},  we also write \eqref{H1} as the following integral test, which will be used later on to apply Proposition~\ref{prop:finiteness}.
\begin{lemma}
	\label{lem:integralTest}
	For all $r \geq 0$, we set ${P}(r) = \left\{ \x \in \mathcal{P} : \bar{Y}(\x) \leq e^{r/3}+1\right\}$. We have
	\begin{equation*}
	\eqref{H1} \quad \Longrightarrow \quad \int_0^\infty \int_{{P}(r)^c} r Y(\x) + \tilde{Y}(\x)~ \Lambda(\dd\x) \dd r < \infty.
	\end{equation*}
\end{lemma}

\begin{proof}
	We observe that, by the Fubini-Tonelli theorem
	\begin{align*}
	&\int_0^\infty \int_{{P}(r)^c} r Y(\x) + \tilde{Y}(\x) \Lambda(\dd\x) \dd r\\
	= &\int_{\mathcal{P}} \int_0^{3 \log_+( \bar{Y}(\x) - 1)} r Y(\x) + \tilde{Y}(\x) \dd r \Lambda(\dd \x)\\
	= &\tfrac{9}{2}\int_\mathcal{P} Y(\x) \log_+(\bar{Y}(\x) - 1)^2 \Lambda(\dd \x) + 3 \int_\mathcal{P} \tilde{Y}(\x) \log_+(\bar{Y}(\x)-1) \Lambda(\dd \x).
	\end{align*}
	By Lemma~\ref{lem:H1bis}, the finiteness of the first integral is equivalent to \eqref{H1}.

	It remains to prove that the second term is finite under \eqref{H1}.  By the same arguments as in the proof of Lemma~\ref{lem:H1bis}, we may assume that $\Lambda$ is a probability measure. 
	Then the finiteness of the second term follows from \cite[Lemma B.1]{Aid13}.  
\end{proof}

\subsection{The sufficient part}
\label{sec:sufficient}

Let $X$ be a branching Lévy process satisfying \eqref{eqn:H0} and \eqref{eqn:var} and $Z^b$  its truncated derivative martingale defined in \eqref{eqn:defTruncatedDerivativeMartingale}. In this section, we assume that \eqref{H1} holds and prove that $Z^b$ is uniformly integrable for all $b > 0$, which by Corollary~\ref{cor:ncandsc} is enough to deduce the sufficient part of Theorem~\ref{thm:main}.

\begin{lemma}
	\label{lem:ZaUI}
	Under assumption \eqref{H1}, $Z^b$ is uniformly integrable for all $b > 0$.
\end{lemma}

\begin{proof}
	By Fact~\ref{fct}, to prove that $(Z^b_t, t \geq 0)$ is uniformly integrable, it is enough to prove that $\Q^b( \liminf_{t \to \infty} Z^b_t < \infty) =1$. By Lemma~\ref{lem:spinalDerivative}, we can equally prove that $\hat{\Q}^b(\liminf_{t \to \infty} Z^b_t \!<\! \infty) \!=\!1$. We write $\mathcal{G} = \sigma(\hat{\xi},\hat{N})$ for the sigma-algebra generated by the spine $\hat{\xi}$ and the birth times and places of its children. 
	By the conditional Fatou lemma, we have 
\[
\E_{\hat{\Q}^b} \left( \liminf_{t \to \infty} Z^b_t \,\middle|\, \mathcal{G}\right) \leq \liminf_{t \to \infty} \E_{\hat{\Q}^b} \left( Z^b_t \,\middle|\, \mathcal{G}\right),\quad \hat{\Q}^b\text{-a.s.} . 
\]
	To conclude the proof, it is therefore enough to show that
	\begin{equation}
	\label{eqn:aimui}
 \liminf_{t \to \infty} \E_{\hat{\Q}^b} \left( Z^b_t \,\middle|\, \mathcal{G}\right) < \infty \quad \hat{\Q}^b\text{-a.s.}.
	\end{equation}

	By the spinal decomposition of $\hat{\Q}^b$ and Lemma~\ref{lem:truncatedDerivativeMartingale}, we have the following $\hat{\Q}^b$-a.s.\ identity:
	\begin{multline*}
	\E_{\hat{\Q}^b} \left( Z^b_t \,\middle|\,  \mathcal{G} \right) = R(b + \hat{\xi}_t) e^{-\hat{\xi}_t}
	+ \int_{[0,t]\times \calP \times \N } \sum_{i\ne k} R(b+ \hat{\xi}_{s-} + x_i)  e^{-  (\hat{\xi}_{s-} + x_i) }  \hat{N}(\dd s, \dd \x , \dd k).
	\end{multline*}
	As $\hat{\xi}$ under $\hat{\Q}^b$ is a centred Lévy process conditioned to stay above $-b$, 
	it is transient, i.e.\ $\lim_{t \to \infty} \hat{\xi}_t = \infty$ $\hat{\Q}^b$-a.s.. Therefore, letting $t \to \infty$, $\hat{\Q}^b$-a.s.\ we have:
	\begin{equation}
	\label{eqn:aimui1}
	\liminf_{t\to \infty} \E_{\hat{\Q}^b}\! \left( Z^b_t \,\middle|\,  \mathcal{G} \right) = \int_{[0,\infty) \!\times\! \calP \!\times\! \N } \sum_{i\ne k} R(b\!+\! \hat{\xi}_{s-} \!+\! x_i)  e^{-  (\hat{\xi}_{s-} \!+\! x_i) }  \hat{N}(\dd s, \dd \x , \dd k).
	\end{equation}
	We divide the above integral into two parts, depending on the relative values of $\x$ and $\hat{\xi}_{t-}$ at the atom $(t,\x,k)$ of $\hat{N}$. More precisely, recalling that
	\[
	\bar{Y}(\x) = \sum_{j \ge 1} \left(1 + x_j \ind{x_j>0}\right) e^{-x_j} \quad \text{and} \quad P(r) = \left\{ \x \in \mathcal{P} : \bar{Y}(\x) \leq e^{r/3}+1\right\},
	\]
	we set
	\begin{align*}
	A^{(1)} &:= \int_{[0,\infty)\times \calP \times \N } \ind{\x \in P(\hat{\xi}_{s-}+b)} \sum_{i\ne k} R(b+ \hat{\xi}_{s-} + x_i)  e^{-  (\hat{\xi}_{s-} + x_i) }  \hat{N}(\dd s, \dd \x , \dd k),\\
	A^{(2)} &:= \int_{[0,\infty)\times \calP \times \N } \ind{\x \not\in P(\hat{\xi}_{s-}+b)}
	\sum_{i\ne k} R(b+ \hat{\xi}_{s-} + x_i)  e^{-  (\hat{\xi}_{s-} + x_i) }  \hat{N}(\dd s, \dd \x , \dd k),
	\end{align*}
	so that \eqref{eqn:aimui1} can be rewritten as $\liminf_{t\to \infty} \E_{\hat{\Q}^b} \left( Z^b_t \,\middle|\,  \mathcal{G} \right) = A^{(1)} \!+\! A^{(2)}$ $\hat{\Q}^b$-a.s.. We now prove that $A^{(1)}$ and $A^{(2)}$ are both $\hat{\Q}^b$-a.s.\@ finite.
	
	An direct application of Lemma~\ref{lem:2} leads to 
	\[
	\E_{\hat{\Q}^b}\left( A^{(1)} \right) =  \int_{[0,\infty)} \E_{\hat{\Q}^b} \left( e^{-\hat{\xi}_s} \,h^{(1)} (\hat{\xi}_s\!+\!b) \right) \dd s,
	\]
	where we set
	\[
	h^{(1)}(y) = \frac{1}{R(y)}\int_{P(y)} \bigg( \sum_{k \geq 1} \sum_{j \neq k} R(y+x_k) R(y+x_j) e^{-x_k-x_j} \bigg) \Lambda(\dd \x).
	\]
	By \eqref{eqn:estimatesRenewal} and the subadditivity of $x \mapsto x \ind{x > 0}$, for all $x \in \R$ and $y \geq 0$ we have $R(x+y) \leq c_2 ( 1 + y )(1 + x\ind{x > 0})$. Therefore, 
	\begin{align*}
	h^{(1)}(y) &\leq C\frac{(1  +y)^2}{R(y)} \int_{P(y)}  \sum_{k \geq 1} \sum_{j \neq k} (1 + x_k\ind{x_k>0})(1 + x_j\ind{x_j>0}) e^{-x_k-x_j} \Lambda(\dd \x)\\
	&\leq C\frac{(1  +y)^2}{y} \int_{P(y)} \left( \left( 1 + x_1 \ind{x_1>0}\right)e^{-x_1} \bar{Y}_2(\x) + \bar{Y}_2(\x) \bar{Y}(\x)\right) \Lambda(\dd \x),
	\end{align*}
	where $\bar{Y}_{2}(\x) = \sum_{k\ge 2} \left(1 + x_k \ind{x_k>0} \right)e^{-x_k} $. As a result, we have
	\begin{align*}
	h^{(1)}(y) &\leq C \frac{(1  +y)^2}{y}\int_{P(y)} (\bar{Y}(\x)+C)\bar{Y}_2(\x) \Lambda(\dd \x)
	\leq C \frac{(1  +y)^2}{y}(e^{y/3}+1) \int_\mathcal{P} \bar{Y}_2(\x) \Lambda(\dd \x).
	\end{align*}
	Since $\int_\mathcal{P} \bar{Y}_2(\x) \Lambda(\dd \x)\!<\!\infty$ by \eqref{eqn:Ybar>2-bis}, we have $h^{(1)}(y) \leq C y^{-1}(1  +y)^2(e^{y/3}\!+\!1)$ and hence $y \mapsto y e^{-y} h^{(1)}(y)$ is integrable on $[0,\infty)$. 
 Combining this fact,  Lemma~\ref{lem:integral} and \eqref{eqn:SpineTruncated}, we conclude that
	\begin{align*}
	\E_{\hat{\Q}^b}(A^{(1)}) &\leq C \int_0^\infty \E_{\hat{\Q}^b} \left( e^{-\hat{\xi}_s} h^{(1)} (\hat{\xi}_s+b) \right) \dd s\leq C e^b \int_0^\infty \E^\uparrow_b \left(  e^{-\xi_s} h^{(1)} (\xi_s)  \right) \dd s < \infty,
	\end{align*}
	which implies that $A^{(1)} < \infty$ $\hat{\Q}^b$-a.s..

	We now turn to $A^{(2)}$. Observe that $A^{(2)}$ is the integral of a random point measure, whose total mass is given by
	\[
	M^{(2)} = \int_{[0,\infty)\times \calP \times \N } \ind{\x \not\in P(\hat{\xi}_{s-}+b)} \hat{N}(\dd s, \dd \x , \dd k).
	\]
	As for all $z \geq 0$, $\sum_{i\ne k} R(z + x_i)  e^{-  (z + x_i) }$ is finite for $\hat{\Lambda}$-almost all $\x$, the finiteness of $A^{(2)}$ is a consequence of the finiteness of $M^{(2)}$. Using first Lemma~\ref{lem:2} and \eqref{eqn:estimatesRenewal} and then \eqref{eqn:SpineTruncated}, we have
	\begin{align*}
	\E_{\hat{\Q}^b} \left(M^{(2)}\right) 
	= &\int_0^{\infty} \E_{\hat{\Q}^b} \Bigg( \int_{ P\left(\hat{\xi}_s+b\right)^{\!c}} \sum_{k\ge 1} \frac{R(b+\hat{\xi}_{s}+x_k)}{R(b + \hat{\xi}_s)} e^{- x_k} \Lambda(\dd\x) \Bigg) \dd s\\
	\le & C \int_0^{\infty} \E_{\hat{\Q}^b}  \left( \frac{1}{R(b\!+\!\hat{\xi}_s)}\int_{ P\left(\hat{\xi}_s+b\right)^{\!c}} \Big((b \!+\! \hat{\xi}_s) Y(\x) + \tilde{Y}(\x)\Big) \Lambda(\dd\x) \right) \dd s\\
	\le & C \int_0^{\infty} \E^{\uparrow}_b \left(\frac{1}{{\xi_s}}  \int_{P\left(\xi_s\right)^{\!c}} \left( \xi_s Y(\x) + \tilde{Y}(\x) \right)\Lambda(\dd\x) \right) \dd s.
	\end{align*}
	By Lemma~\ref{lem:integralTest}, under assumption \eqref{H1} we have 
	\[
	\int_0^\infty \int_{P(r)^c} \left(r Y(\x) + \tilde{Y}(\x)\right) \Lambda(\dd \x) \dd r < \infty. 
	\] 
	Since the function $r\mapsto \int_{P(r)^c} \frac{1}{r}\left(r Y(\mathbf{x}) + \tilde{Y}(\mathbf{x})\right) \Lambda(\dd \mathbf{x})$ is bounded by \eqref{eqn:Ybar>2-E} and clearly decreasing, Lemma~\ref{lem:integral} shows that $M^{(2)}$ and hence $A^{(2)}$ are $\hat{\Q}^b$-a.s.\ finite.
	
	We conclude that  $A^{(1)} + A^{(2)} < \infty$ $\hat{\Q}^b$-a.s., from which we deduce by \eqref{eqn:aimui1} that \eqref{eqn:aimui} holds, completing the proof.
\end{proof}

\begin{proof}[Proof of the sufficient part of Theorem~\ref{thm:main}]
	We assume that \eqref{H1} holds. Using Lemma~\ref{lem:ZaUI}, $(Z^b_t)$ is uniformly integrable for all $b > 0$. Therefore, by Corollary~\ref{cor:ncandsc}, we obtain that $Z_\infty$ is non-degenerate, which completes the proof.
\end{proof}

\subsection{The necessary part}
\label{sec:necessary}

In this section, we prove that if \eqref{H1} does not hold, then $Z_{\infty}\!=\!0$ $\P$-a.s.. Taking Corollary~\ref{cor:ncandsc} into account, it suffices to prove that
$Z^b_\infty = 0$ $\P$-a.s.\@ for all $b > 0$. 
By Fact~\ref{fct} and Lemma~\ref{lem:spinalDerivative}, the problem boils down to proving the following lemma.

\begin{lemma}
	\label{lem:bc}
	If \eqref{H1} does not hold, then for every $b>0$, 
	\[
	\limsup_{t \to \infty} Z^b_t = \infty, \quad \hat{\Q}^b\text{-a.s..}
	\]
\end{lemma}

Its proof relies crucially on a conditional version of the Borel-Cantelli lemma, obtained by \cite{Che78}, for a sum of non-negative random variables. A simplified version of this result can be stated as follows. 
\begin{lemma}[{\cite{Che78}}]
\label{lem:LouisChen}
Let $(V_n, n\ge 1)$ be a sequence of non-negative random variables defined on a probability space $(\Omega, \mathcal{H}, \mathbf{P})$ and $(\mathcal{H}_n, n\ge 0)$ a filtration. Let $U_n= \mathbf{E}(V_n\mid \mathcal{H}_{n-1})$, $n\ge 1$. 
\begin{enumerate}
\item $\sum_{n\ge 1} V_n <\infty$ a.s.\  on $\{\sum_{n\ge 1} U_n <\infty \}$. 
\item If $(V_n)$ is bounded and $(\mathcal{H}_n)$-adapted, then $\sum_{n\ge 1}\! U_n <\infty $ a.s.\ on $\big\{\sum_{n\ge 1}\! V_n \!<\!\infty \big\}$.  
\end{enumerate}
\end{lemma}

\begin{proof}[Proof of Lemma~\ref{lem:bc}]
	Fix $b > 0$. We prove this lemma by contraposition, showing that 
	\[
	\hat{\Q}^b\Big(\limsup_{t \to \infty} Z^b_t < \infty \Big)>0  \quad \Rightarrow \quad \eqref{H1}. 
	\]
	 Recall from Lemmas~\ref{lem:spinalDecomposition} and \ref{lem:spinalDerivative} the construction of $\hat{\Q}^b$ from $\hat{\P}$ via a change of probabilities and the spinal decomposition for both. With notation therein, for each atom $(t,\x,k)$ of $\hat{N}$,  by \eqref{eqn:estimatesRenewal} we have
	\[
	Z^b_t \geq \sum_{j \geq 1} R(b + \hat{\xi}_{t-} + x_j) e^{-(\hat{\xi}_{t-}+x_j)} \geq c_1 e^{-\hat{\xi}_{t-}}\sum_{j \geq 1} \left(b + \hat{\xi}_{t-} + x_j\right)_+  e^{-x_j}.
	\]
	The assumption $\hat{\Q}^b(\limsup_{t \to \infty} Z^b_t < \infty)>0$ implies that there exist $T,A > 0$ such that
	\begin{equation}
	\label{eqn:firstStep}
	\hat{\Q}^b\bigg(\forall (t,\x,k) \text{ atom of } \hat{N} : t \!\leq\! T \text{ or } \sum_{j \geq 1} \Big(b \!+\! \hat{\xi}_{t-}\! +\! x_j\Big)_{\!\!+}  \!e^{-x_j} \!\leq\! A e^{\hat{\xi}_{t-}}\bigg) > 0.
	\end{equation}
	
	Write $\hat{\mathcal{F}}$ for the natural filtration of the branching Lévy process with spine.
	Let 
	\[
	G:= \bigg\{ (t,\x,k) \text{ atoms of } \hat{N}\colon \sum_{j \geq 1} \left(b + \hat{\xi}_{t-} + x_j\right)_{\!+}  e^{-x_j} > A e^{\hat{\xi}_{t-}+b}\bigg\}.
	\]
	Since $\sum_{j \geq 1} \left(b + y + x_j\right)_+  e^{-x_j}\le \max(b + y, 1) \bar{Y} (\x)$, we may assume that $A$ is large enough such that $G\subset \{(t,\x,k) \text{ atoms of } \hat{N}\colon \bar{Y}(\x) > 2\}$. 
	By \eqref{eqn:Ybar>2-E}, we have $\hat{\Lambda}(\bar{Y}(\x) > 2) < \infty$, where $\hat{\Lambda}$ is defined by \eqref{eqn:hatLambdaDef}. So the first coordinates of the atoms in $G$ cannot accumulate in finite time $\hat{\P}$-a.s.. As $\hat{\Q}^b_{|\hat{\mathcal{F}}_t}$ is absolutely continuous with respect to $\hat{\P}_{|\hat{\mathcal{F}}_t}$ for any fixed time $t>0$, they cannot accumulate in finite time $\hat{\Q}^b$-a.s.\ either. Therefore, \eqref{eqn:firstStep} implies that
	\begin{equation}
	\label{eqn:S}
	\hat{\Q}^b(\# G<\infty) >0. 
	\end{equation}
	
	We now prove that condition \eqref{eqn:S} implies \eqref{H1}. To use Lemma~\ref{lem:LouisChen}, we need a time-discretization argument. Specifically, enumerate  the atoms of $\hat{N}$ such that $\bar{Y}(\x) > 2$ in increasing order of time by $\{(\tau_n, \x_n, k_n),n\ge 1\}$. 
	Then $(\tau_n)$ are $\hat{\mathcal{F}}$-stopping times that do not accumulate $\hat{\mathbb{Q}}^b$-a.s.. Set $\tau_0=0$ and 
	\[
	B_n = \sum_{(t,\x,k) \text{ atom of } \hat{N}}   \ind{t \in [\tau_n ,\tau_{n+1})}  \ind{ \sum_{j \geq 1} \left(b + \hat{\xi}_{t-} + x_j\right)_+  e^{-x_j}> A e^{\hat{\xi}_{t-}+b}},\qquad n\ge 0.
	\]
	Then we have $\#G= \sum_{n \geq 1} B_n$. 
	Note that each $B_n$ is $\hat{\mathcal{F}}_{\tau_n}$-measurable and takes values in $\{0,1\}$.
	
	As the family $(B_n)$ is bounded, it follows from the second part of Lemma~\ref{lem:LouisChen} that
	\begin{align}
	\hat{\Q}^b\Big(\sum_{n \geq 1} B_n < \infty\Big)
	= 
	\hat{\Q}^b\Big(\sum_{n\ge 1} \E_{\hat{\Q}^b}\left(B_n\,\middle|\,\hat{\mathcal{F}}_{\tau_{n-1}} \right) < \infty\Big)\label{eq:sum-Bn}.
	\end{align}
	We use the Markov property and Lemma \ref{lem:2} to obtain
	\begin{equation}
	\label{eq:B_i-conditioned}
	\E_{\hat{\Q}^b}\big(B_n\,\big|\,\hat{\mathcal{F}}_{\tau_{n-1}} \big)  = \E_{\hat{\Q}^b}\bigg(\int_{\tau_{n-1}}^{\tau_n} h(\hat{\xi}_r)\dd r \,\bigg|\, \hat{\mathcal{F}}_{\tau_{n-1}}\bigg), \quad \hat{\Q}^b\text{-a.s.}, 
	\end{equation}
	where the function $h\colon \R\to \R$ is defined by
	\[
	h(y) := \int_{\mathcal{P}} \ind{ \sum_{j \geq 1} \left(b + y + x_j\right)_+  e^{-x_j}> A e^{y+b}}\sum_{k \geq 1} e^{-x_k}\frac{R(b+y+x_k)}{R(b+y)}  \Lambda( \dd \x).
	\]
	We deduce by \eqref{eqn:S}, \eqref{eq:sum-Bn} and \eqref{eq:B_i-conditioned} that 
	\begin{equation*}
	\hat{\Q}^b \bigg( \sum_{n\ge 1}\E_{\hat{\Q}^b}\bigg(\int_{\tau_{n-1}}^{\tau_n} h(\hat{\xi}_r)\dd r \,\bigg|\, \hat{\mathcal{F}}_{\tau_{n-1}} \bigg) < \infty\bigg)>0.
	\end{equation*}	
	Then, using the first part of Lemma~\ref{lem:LouisChen}, which does not require the  boundedness, we deduce that 
	\begin{equation}\label{eqn:h-finite}
	\hat{\Q}^b \bigg( \int_{0}^{\infty} h(\hat{\xi}_r)\dd r   < \infty\bigg)>0.
	\end{equation}
	Define a function $f\colon \R\to \R$ by
	\[
	f(y) := \int_{\mathcal{P}} \ind{y>0} \ind{ \sum_{j \geq 1} \!(y + x_j)_{\!+}  \!e^{-x_j}> A e^{y}}\sum_{k \geq 1} e^{-x_k}\frac{(y+x_k)_+}{y+1}  \Lambda( \dd \x).
	\]
	By \eqref{eqn:estimatesRenewal}, there exists $c > 0$ such that $h(\,\cdot\,)\ge c f(\,\cdot\,+b)$. Therefore, \eqref{eqn:h-finite} leads to:   
	\begin{equation}
	\label{eqn:int-xi1}
	\hat{\Q}^b \bigg( \int_{0}^{\infty} f(\hat{\xi}_r+b)\dd r<\infty\bigg)>0, \text{ which is equivalent to }  \P^{\uparrow}_b\bigg( \int_{0}^{\infty} f(\xi_r)\dd r<\infty\bigg)>0,
	\end{equation} 
    in view of \eqref{eqn:SpineTruncated}.
	To directly apply Proposition~\ref{prop:finiteness} to the perpetual integral, we would need $f$ to be eventually non-increasing, which is not necessarily the case. Instead, we use various non-increasing lower bounds of $f$ to show that the right-hand side of \eqref{eqn:int-xi1} successively implies the three following integral tests:
	\begin{align}
	&\int_{\mathcal{P}} \tilde{Y}(\x) \log_+ (\tilde{Y}(\x)-1) \Lambda(\dd \x)<\infty,\label{eqn:Int1} \\
	&\int_{\mathcal{P}} Y(\x) \log_+ (Y(\x)-1) \Lambda(\dd \x)<\infty, \label{eqn:Int2}\\
	&\int_{\mathcal{P}} Y(\x) \log_+ (Y(\x)-1)^2 \Lambda(\dd \x)<\infty.\label{eqn:Int3}
	\end{align}
	Note that \eqref{eqn:Int1} and \eqref{eqn:Int3} immediately imply \eqref{H1}, which will complete the proof. Equation \eqref{eqn:Int2} is needed as a first step to prove \eqref{eqn:Int3}.
	
	We first prove \eqref{eqn:Int1}. For $y \geq 0$, we have $(y + x_j)_+ \geq (x_j)_+$. Therefore for all $y \geq 0$,
	\[
	f(y) \geq f_{1}(y) := \int_{\mathcal{P}} \ind{y>0}\ind{ \sum_{j \geq 1} \!(x_j)_{\!+}  \!e^{-x_j}> A e^{y}}\sum_{k \geq 1} e^{-x_k}\frac{(x_k)_+}{y+1}  \Lambda( \dd \x).
	\]
	As $f_{1}$ is non-increasing and \eqref{eqn:int-xi1} implies $\P^{\uparrow}_b \left( \int_{0}^{\infty} f_{1}(\xi_r)\dd r<\infty\right)>0$, using Proposition \ref{prop:finiteness} and Fubini's theorem, we deduce that
\begin{align*}	
	\infty > \int_0^\infty (y+1) f_{1}(y) \dd y
&	= \int_0^\infty \dd y \int_{\mathcal{P}}\ind{ \sum_{j \geq 1} \!(x_j)_{\!+}  \!e^{-x_j}> A e^{y}}\sum_{k \geq 1} e^{-x_k} (x_k)_+  \Lambda( \dd \x)\\
&=	\int_{\mathcal{P}} \tilde{Y}(\x) \log_+(\tilde{Y}(\x) / A) \Lambda(\dd \x). 
	\end{align*}
	This yields \eqref{eqn:Int1}, since we have 
\begin{multline}	
 \bigg|\int_{\mathcal{P}}\tilde{Y}(\x) \log_+(\tilde{Y}(\x)\! -\! 1) \Lambda(\dd \x) - \int_{\mathcal{P}} \tilde{Y}(\x) \log_+(\tilde{Y}(\x) / A) \Lambda(\dd \x)\bigg|\\
   \le A \log A \cdot \Lambda\Big(A\ge \tilde{Y}(\x)\geq 2\Big)+ \log A \int_{\mathcal{P}} \ind{\tilde{Y}(\x)> A}\tilde{Y}(\x)  \Lambda(\dd \x) <\infty, \label{eq:diff-A-1}
	\end{multline}
	where the finiteness of the last two terms are from \eqref{eqn:Ybar>2} and \eqref{eqn:Ybar>2-E}.

	We now turn to the proof of \eqref{eqn:Int2}. 
	To this end, for each $y \geq 0$, we divide $Y(\x)$ into
	\[
	Y^y_+(\x) := \sum_{j \geq 1} \ind{x_j \geq -y/2} e^{-x_j} \quad \text{and} \quad Y^y_-(\x):= \sum_{j \geq 1} \ind{x_j < -y/2} e^{-x_j}.
	\]
	Notice that for all $\x \in \mathcal{P}$, $L \in \R$ and $y \in \R_+$, we have
	\begin{align*}
	Y(\x) \ind{Y(\x) \geq L}
	&\!= \left(Y^y_+(\x) + Y^y_-(\x)\right)\ind{Y(\x) \geq L}\\
	&\!= Y^y_+(\x) \ind{Y^y_+(\x) \geq L} \!+\! Y^y_+(\x) \ind{Y(\x) \geq L \geq Y^y_+(\x)} \!+\! Y^y_-(\x) \ind{Y(\x) \geq L}.
	\end{align*}
	Additionally, since $Y(\x) \geq Y^y_+(\x) \geq Y(\x)/2$ whenever $Y^y_+(\x) \geq Y^y_-(\x)$, we have 
	\begin{align*}
	Y^y_+(\x) \ind{Y(\x) \geq L \geq Y^y_+(\x)}
	\! &\leq Y^y_+(\x) \ind{Y(\x) \geq L \geq Y^y_+(\x) \geq Y^y_-(\x)} \!+Y^y_+(\x) \ind{Y^y_+(\x) < Y^y_-(\x)}\\
	&\leq Y(\x) \ind{Y(\x) \geq L \geq Y(\x)/2} + Y^y_-(\x),
	\end{align*}
 As a result, we have
	\begin{equation}
	\label{eqn:decompositionOfY}
	Y(\x) \ind{Y(\x) \geq L} \leq Y^y_+(\x) \ind{Y^y_+(\x) \geq L} + Y(\x) \ind{Y(\x) \geq L \geq Y(\x)/2} + 2Y^y_-(\x).
	\end{equation}

	By similar arguments as in \eqref{eq:diff-A-1}, to show that \eqref{eqn:Int2} holds, it is enough to prove that 
	\[
	\int_{\mathcal{P}} Y(\x) (\log(Y(\x)/A) - 2)_+ \Lambda(\dd \x) = \int_2^\infty \dd y \int_{\mathcal{P}} Y(\x) \ind{Y(\x) \geq Ae^y} \Lambda(\dd \x)<\infty,
	\]
	 where the equality is due to Fubini's theorem. 
	By Proposition~\ref{prop:finiteness}, this is equivalent to 
	\[
	\P^{\uparrow}_b\bigg( \int_{2}^{\infty}f_2(\xi_r)\dd r <\infty\bigg)>0,\quad  \text{where}~ f_2(y):=\frac{1}{y+1} \int_{\mathcal{P}} Y(\x) \ind{Y(\x) \geq Ae^{y}} \Lambda(\dd \x).
	\]
	In view of \eqref{eqn:decompositionOfY}, it suffices to prove that
	\begin{align}
	\label{eqn:aim1}
	&\P^{\uparrow}_b\bigg( \int_{2}^{\infty}f_3(\xi_r)\dd r <\infty\bigg)>0,\quad\text{where}~ f_3(y):= \frac{1}{y+1} \int_{\mathcal{P}} Y^y_+(\x) \ind{Y^y_+(\x) \geq Ae^y} \Lambda(\dd \x) ,\\
	\label{eqn:aim2}
	&\E^{\uparrow}_b\bigg( \int_{2}^{\infty}f_4(\xi_r)\dd r\bigg) <\infty, \quad \text{where}~f_4(y):=\frac{1}{y+1}\int_{\mathcal{P}}Y(\x) \ind{Y(\x) \geq Ae^y \geq Y(\x)/2}  \Lambda(\dd \x),  \\
	\label{eqn:aim3}
	&\E^{\uparrow}_b\bigg( \int_{2}^{\infty}f_5(\xi_r)\dd r\bigg) <\infty, \quad\text{where}~ f_5(y):= \frac{1}{y+1} \int_{\mathcal{P}} Y^y_-(\x) \Lambda(\dd \x).
	\end{align}
	Using Fubini's theorem and Lemma~\ref{lem:integral}, we have 
		\begin{align*}
\E^{\uparrow}_b\bigg( \int_{2}^{\infty}f_4(\xi_r)\dd r\bigg)
&=  \int_{\mathcal{P}}Y(\x)\Lambda(\dd \x) \E^{\uparrow}_b\bigg( \int_2^\infty \frac{1}{\xi_r+1} \ind{Y(\x) \geq Ae^{\xi_r} \geq Y(\x)/2}\dd r  \bigg) \\
& \le C \int_{\mathcal{P}}Y(\x)\Lambda(\dd \x)  \int_2^\infty \frac{y}{y+1} \ind{Y(\x) \geq Ae^{y} \geq Y(\x)/2}\dd y  \\
&\le C \log 2 \int_{\mathcal{P}}Y(\x)\ind{Y(\x)\ge A}\Lambda(\dd \x).
	\end{align*}
	The last integral is finite by \eqref{eqn:Ybar>2-E}, so \eqref{eqn:aim2} follows. 
	Similarly, we deduce \eqref{eqn:aim3} by using Lemma~\ref{lem:integral} and the observation that 
	\[
	\int_2^\infty \dd y \int_{\mathcal{P}} Y^y_-(\x) \Lambda(\dd \x) \leq 2 \int_{\mathcal{P}} \sum_{j \geq 1} |x_j| e^{-x_j} \ind{x_j < -1} \Lambda(\dd \x),
	\]
	which is finite by \eqref{eqn:varbis}. 
	
	We now turn to \eqref{eqn:aim1}. For all $y\geq 2$, using the fact that $(x_j+y)_+ \geq y\ind{x_j > -y/2}/2 \geq \ind{x_j > -y/2}$, we obtain
	\begin{equation*}
	f(y) \geq \frac{1}{(y+1)} \int_{\mathcal{P}}\ind{  Y^y_+(\x) > A e^{y}} Y_+^y(\x)  \Lambda( \dd \x) = f_3(y).
	\end{equation*}
Now \eqref{eqn:aim1} follows from \eqref{eqn:int-xi1}. 
	This completes the proof of \eqref{eqn:Int2}.
	
	Finally, we turn to the proof of \eqref{eqn:Int3}, by following similar steps as the proof of \eqref{eqn:Int2}. Using Fubini's theorem, \eqref{eqn:Ybar>2-E} and Proposition~\ref{prop:finiteness}, we observe that \eqref{eqn:Int3} is a consequence of
	\begin{align*}
			&\int_2^\infty y \dd y  \int_{\mathcal{P}} \ind{Y(\x) > Ae^y} Y(\x) \Lambda(\dd \x) < \infty 	\\
			\iff& \P^{\uparrow}_b\bigg( \int_{2}^{\infty}\dd r \int_{\mathcal{P}} \ind{Y(\x) > Ae^{\xi_r}} Y(\x) \Lambda(\dd \x) <\infty\bigg)>0, 
	\end{align*}
	which, using again \eqref{eqn:decompositionOfY}, is a consequence of the three conditions
	\begin{align}
	\label{eqn:bisaim1}
	&\P^{\uparrow}_b\bigg( \int_2^\infty f_6(\xi_r) \dd r< \infty\bigg)>0, \quad \text{where}~ f_6(y) := \int_{\mathcal{P}} Y^y_+(\x) \ind{Y^y_+(\x) \geq Ae^y} \Lambda(\dd \x),\\
	\label{eqn:bisaim2}
	&\E^{\uparrow}_b\bigg( \int_2^\infty f_7(\xi_r) \dd r\bigg)< \infty, \quad \text{where}~ f_7(y) :=  \int_{\mathcal{P}}Y(\x) \ind{Y(\x) \geq Ae^y \geq Y(\x)/2}  \Lambda(\dd \x),\\
	\label{eqn:bisaim3}
	&\E^{\uparrow}_b\bigg( \int_2^\infty f_8(\xi_r) \dd r\bigg)< \infty, \quad \text{where}~ f_8(y) := \int_{\mathcal{P}} Y^y_-(\x) \Lambda(\dd \x) < \infty.
	\end{align}	
	We remark that 
	\[
	\int_2^\infty \!\! y \dd y \int_{\mathcal{P}}Y(\x) \ind{Y(\x) \geq Ae^y \geq Y(\x)/2}  \Lambda(\dd \x) \leq C \int_{\mathcal{P}} \log_+(Y(\x)/A-2) Y(\x) \Lambda(\dd \x),
	\]
	which is finite as we have already proved \eqref{eqn:Int2},  on account of \eqref{eqn:Ybar>2-E}. Similarly, we have
	\[
	\int_2^\infty y \dd y \int_{\mathcal{P}} Y^y_-(\x) \Lambda(\dd \x) \leq \int_{\mathcal{P}} \sum_{j \geq 1} x_j^2 e^{-x_j} \ind{x_j < -1} \Lambda(\dd \x),
	\]
	which is finite by \eqref{eqn:varbis}. Then we can deduce \eqref{eqn:bisaim2} and \eqref{eqn:bisaim3}, by using Fibini's theorem and Lemma~\ref{lem:integral}, in the same way as we prove \eqref{eqn:aim2}.

	We thus only have to prove \eqref{eqn:bisaim1}.
	For all $y \geq 2$, since $(y + x_j)_+\ge \ind{x_j \geq -y/2}(y+1)/3$, we have
	\[
	f(y) \geq  \frac{1}{3} \int_{\mathcal{P}}\ind{  Y^y_+(\x) > A e^{y}} Y_+^y(\x)  \Lambda( \dd \x).
	\]
	Then \eqref{eqn:bisaim1} follows from \eqref{eqn:int-xi1}, which completes the proof of  \eqref{eqn:Int3}. 
	
	In sum, assuming that $\hat{\Q}^b\left(\limsup_{t \to \infty} Z^b_t < \infty\right)>0$, we deduce \eqref{eqn:Int1}, \eqref{eqn:Int2} and \eqref{eqn:Int3}. Therefore, \eqref{H1} holds. 
\end{proof}

\begin{proof}[Proof of the necessary part of Theorem~\ref{thm:main}]
	Assume that \eqref{H1} does not hold. By Lemma~\ref{lem:bc}, we have
	\[
	\limsup_{t \to \infty} Z^b_t = \infty, \quad \hat{\Q}^b\text{-a.s.,} \quad \text{ for all } b > 0. 
	\]
	By Fact~\ref{fct}, this implies that $\lim_{t \to \infty} Z^b_t = 0$, $\P$-a.s.. Using Corollary~\ref{cor:ncandsc}, we conclude that $Z_\infty = 0$ $\P$-a.s., completing the proof.
\end{proof}

%

%
%

\begin{acks}[Acknowledgments]
 The authors would like to thank Samuel Baguley, Loïc Chaumont, Leif Döring and Andreas Kyprianou for valuable inputs in the proof of Proposition~\ref{prop:finiteness}. 
 We also thank Haojie Hou and an anonymous referee for their helpful comments. 
\end{acks}
\begin{funding}
	Part of this work was supported by a PEPS JCJC grant.
BM is partially supported by the Projet ANR-16-CE93-0003 (ANR MALIN). QS was partially supported by the SNSF Grant P2ZHP2\_171955. 
\end{funding}



\bibliographystyle{imsart-number} 
\bibliography{biblio}       


\end{document}